\documentclass{amsart}
\usepackage{amsrefs}
\numberwithin{equation}{section}

\def \rn {\mathbb{R}^n}

\def \rr {\mathbb{R}}
\def \nn {\mathbb{N}}
\def \crit {2^\star}
\def \crits {2^\star(s)}
\def \eps {\epsilon}
\def \xe {x_\epsilon}
\def \ue {u_\epsilon}

\def \we {w_\epsilon}
\def \lae {\lambda_\epsilon}

\def \bu {\bar{u}}

\def \ballp {B_1(0)\setminus\{0\}}
\def \rnp {\rn\setminus\{0\}}
\newtheorem{thm}{Theorem}
\newtheorem{prop}{Proposition}[section]
\newtheorem{defi}{Definition}
\newtheorem{coro}{Corollary}[section]
\newtheorem{lem}{Lemma}[section]

\def\neweq#1{\begin{equation}\label{#1}}
\def\endeq{\end{equation}}

\title[Sharp profiles of singular solutions to elliptic equations]{Sharp asymptotic profiles for singular solutions to an elliptic equation with a sign-changing nonlinearity}
\author{Florica C. C\^{\i}rstea}
\address{School of Mathematics and Statistics, The University of Sydney, NSW 2006, Australia}
\email{florica.cirstea@sydney.edu.au}
\author{Fr\'ed\'eric Robert}
\address{Institut Elie Cartan de Lorraine, Universit\'e de Lorraine, BP 70239, 54506 Vand\oe uvre-l\`es-Nancy, France}
\email{frederic.robert@univ-lorraine.fr}
\date{January 20th, 2016}
\begin{document}

\thanks{Part of this work was performed while FR was visiting the University of Sydney. He thanks FC and the School of Mathematics and Statistics for their support. FR is supported by a grant from ``Universit\'e de Lorraine" and ``R\'egion Lorraine". FC was supported by ARC Discovery grant number DP120102878 ``Analysis of non-linear partial differential equations describing singular phenomena''. FC is also grateful for the support received during her visits to University of Lorraine in 2014 and 2016.  FC and FR gratefully acknowledge the support from the programme PHC FAST 12739WA ``Progress in geometric analysis and applications" coordinated by Philippe Delano\"e and Neil Trudinger. }
\subjclass[2010]{Primary 35J91; Secondary 35A20, 35J75, 35J60}

\begin{abstract} Given $B_1(0)$ the unit ball of $\rn$ ($n\geq 3$), we study smooth positive singular solutions $u\in C^2(B_1(0)\setminus \{0\})$ to $-\Delta u=\frac{u^{\crits-1}}{|x|^s}-\mu u^q$. Here $0< s<2$, $\crits:=2(n-s)/(n-2)$ is critical for Sobolev embeddings, $q>1$ and $\mu> 0$. When $\mu=0$ and $s=0$, the profile at the singularity $0$ was fully described by Caffarelli-Gidas-Spruck. We prove that when $\mu>0$ and $s>0$, besides this profile, two new profiles might occur. We provide a full description of all the singular profiles. Special attention is accorded to solutions such that $\liminf_{x\to 0}|x|^{\frac{n-2}{2}}u(x)=0$ and $\limsup_{x\to 0}|x|^{\frac{n-2}{2}}u(x)\in (0,+\infty)$. The particular case $q=(n+2)/(n-2)$ requires a separate analysis which we also perform. 
\end{abstract}
\maketitle
\tableofcontents
\section{Introduction}
We let $B_1(0)$ be the unit ball of $\rn$ with $n\geq 3$. For $s\in (0,2)$, $q>1$ and $\mu>0$ fixed, we consider a positive function $u\in C^\infty(\ballp)$ such that
\begin{equation}\label{eq:u}
-\Delta u=\frac{u^{\crits-1}}{|x|^s}-\mu u^q\hbox{ in }\ballp,
\end{equation}
where $\crits:=\frac{2(n-s)}{n-2}$ is critical from the viewpoint of the Hardy--Sobolev embeddings. We say that $0$ is a removable singularity for $u$ if $u$ can be extended at $0$ by a H\"older function. Otherwise, we say that $0$ is a non-removable singularity. Our objective here is to analyze the behavior of $u$ at $0$ when $0$ is a non-removable singularity.

\medskip\noindent For the sole pure Sobolev critical nonlinearity, that is when $\mu=s=0$, the equation $-\Delta u=u^{\crit-1}$ is conformally invariant (here, $\crit:=2^\star(0)=\frac{2n}{n-2}$). In this context, the pioneering analysis is due to Caffarelli-Gidas-Spruck \cite{cgs}. Using the Alexandrov reflection principle, they showed that singular solutions are controlled from above and below by $x\mapsto |x|^{-\frac{n-2}{2}}$ around $0$. They also outlined the central role of 
\begin{equation}\label{key}
 x\mapsto W(x):=|x|^{\frac{n-2}{2}}u(x).
\end{equation}
More precisely, Caffarelli-Gidas-Spruck proved that, up to a change of variable, the function $W$ (defined in \eqref{key})
behaves around $0$ like a positive periodic function in $\ln |x|$. In the sequel, such a behavior will be referred to as {\it (CGS)} profile (see the precise definition below). The blow-up profile has been refined by Korevaar-Mazzeo-Pacard-Schoen \cite{kmps}. When $\mu=0$ and $s>0$, Hsia-Lin-Wang \cite{hlw} proved that singular solutions to \eqref{eq:u} also blow-up along a (CGS) profile.

\medskip\noindent The situation happens to be much richer when one drifts away from the conformally invariant equation, that is when $\mu>0$ in \eqref{eq:u} (in addition to $s>0$). In equation \eqref{eq:u}, three terms compete with each other: asymptotically, one expects that one of these terms is negligible. In Theorem \ref{thm:1}, we prove that this is the case: moreover, the function $W$ in \eqref{key} discriminates the three regimes of singular solutions.

\smallskip\noindent Indeed, when $\lim_{x\to 0}W(x)=0$, then the singularity is removable. This situation always occurs when $q>\crit-1$.

\smallskip\noindent When $W(x)< C$ around $0$ for some constant $C>0$, then $\mu u^q$ is negligible for the preliminary analysis, and a singular solution $u$ behaves essentially like a smooth positive solution to
\begin{equation}\label{eq:U:rnp}
-\Delta U=\frac{U^{\crits-1}}{|x|^{s} }\hbox{ in }\rnp.
\end{equation}
Here, two potential profiles might occur. When $c<W< C$ around $0$  for some positive constants $c,C>0$, then the classical (CGS) profile occurs: this is the first blow-up profile. However, unlike the exact conformally invariant equation \eqref{eq:U:rnp}, the function $W$ might oscillate between $0$ and a positive constant: in this situation, a second profile occurs, namely the profile of type (MB) (for ``Multi-Bump'') described below. We prove that the existence of this (MB) profile is due to a nontrivial influence of the perturbation $\mu u^q$. A related phenomenon has been observed by Chen-Lin \cite{chenlin} for equation $-\Delta u= K(x) u^{\crit-1}$ with $x\mapsto K(x)$ having a specific behavior at $0$.

\medskip\noindent The introduction of the weight $|x|^{-s}$ in the equation generates a third asymptotic profile. Indeed, unlike the scalar curvature-type equations studied in \cite{cgs}, \cite{kmps} and \cite{hlw}, there are singular solutions to \eqref{eq:u} that are not controled by $x\mapsto |x|^{-\frac{n-2}{2}}$  when $\crits-1<q<\crit-1$, and therefore, $W$ is not even bounded from above. In this situation, we observe that $-\Delta u$ is negligible in \eqref{eq:u} compared to the nonlinear part. In particular, we show that $u$ behaves like the solution to $|x|^{-s} U^{\crits-1}-\mu U^q=0$. We then say that the profile is of (ND) type (for ``Non Differential'').

\medskip\noindent Nonsingular positive solutions to \eqref{eq:U:rnp} are exactly of the form
\neweq{U1-def}x\mapsto U(x)=U_\lambda(x):=\lambda^{-\frac{n-2}{2}} U_1(x/\lambda):=c_{n,s}\left(\frac{\lambda^{\frac{2-s}{2}}}{\lambda^{2-s}+|x|^{2-s}}\right)^{\frac{n-2}{2-s}} \quad \text{in }\rn,\endeq
for some $\lambda>0$, where $c_{n,s}:=\left((n-s)(n-2)\right)^{\frac{1}{\crits-2}}$ (see Proposition \ref{prop:poho:U}).

\begin{defi} We say that $u$ develops a profile of (CGS) type if there exists a positive periodic function $v\in C^\infty(\rr)$ such that
$$\lim_{x\to 0}\left(|x|^{\frac{n-2}{2}}u(x)-v(-\ln |x|)\right)=0.\eqno{(CGS)}$$

\medskip\noindent We say that $u$ develops a profile of (MB) type (for ``Multi-Bump") if there exists a sequence $(r_k)_k>0$ decreasing to $0$ such that $r_{k+1}=o(r_k)$ as $k\to +\infty$ and 
\begin{equation*}
u(x)=\left(1+o(1)\right)\sum_{k=0}^\infty c_{n,s}\left(\frac{r_{k}^{\frac{2-s}{2}}}{r_{k}^{2-s}+|x|^{2-s}}\right)^{\frac{n-2}{2-s}}\hbox{ as }x\to 0.\eqno{(MB)}
\end{equation*}
\medskip\noindent We say that $u$ develops a (ND) type profile (for ``Non Differential") if 
$$\lim_{x\to 0}|x|^{\frac{s}{q-(\crits-1)}}u(x)=\mu^{-\frac{1}{q-(\crits-1)}}.\eqno{(ND)}$$
\end{defi}

\medskip\noindent We are now in position to state our first theorem. We prove that when $q\neq \crit-1$, then singular solutions to \eqref{eq:u} behave according to one of these three profiles.

\begin{thm}\label{thm:1} Let $u\in C^\infty(\ballp)$ be a positive solution to \eqref{eq:u}. Then
 
 \medskip $\bullet$ If $q>\crit-1$, then $0$ is a removable singularity,

\medskip $\bullet$ If $\crits-1<q<\crit-1$, then either $0$ is a removable singularity, or $u$ develops a profile of type (CGS), (MB) or (ND),

\medskip $\bullet$ If $q\leq \crits-1$, then either $0$ is a removable singularity, or $u$ develops a profile of type (CGS) or (MB).

\medskip\noindent Moreover, if $u$ develops an (MB) profile, then $\crit-2<q<\crit-1$ and the sequence $(r_k)$ satisfies 
$$r_{k+1}=(K+o(1))r_{k}^{\frac{1}{q-(\crit-2)}}$$
as $k\to +\infty$, where $K$ is the positive constant defined by
$$K:=\left(\frac{(\crit-1-q)c_{n,s}^{q-1}\mu }{(q+1)(n-2)|\partial B_1(0)|}\int_{\rn} \frac{dx}{(1+|x|^{2-s})^{(q+1)\frac{n-2}{2-s}}}\right)^{\frac{2}{(n-2)(q-(\crit-2))}}.$$
\end{thm}

\newpage\noindent The characterization of the three blow-up profiles is summarized in this table:

\begin{table}[ht]

\begin{center}
\begin{tabular}{|c|c|c|c|}
\hline
{\bf Type}  &  $\liminf_{x\to 0}|x|^{\frac{n-2}{2}}u(x)$ & $\limsup_{x\to 0}|x|^{\frac{n-2}{2}}u(x)$ \\

\hline
removable &   $0$ & $0$\\
\hline 
(CGS) &    $\in (0,\infty)$ & $\in (0,\infty)$ \\
\hline
(MB) &   $0$ & $\in (0,\infty)$\\
\hline 
(ND) &    $\infty$ & $\infty$\\
\hline
\end{tabular}
\bigskip
\end{center}
\end{table}

\noindent{\it Remark:} From the analysis viewpoint, it is more convenient to express the asymptotic behavior (MB) in the following equivalent form: for any $R>0$, for any $x\in B_{R r_k}(0)\setminus B_{R^{-1}r_{k+1}}(0)$, we have that
\begin{equation*}
u(x)=\left(1+\eps_k(x)\right)\left(c_{n,s}\left(\frac{r_{k+1}^{\frac{2-s}{2}}}{r_{k+1}^{2-s}+|x|^{2-s}}\right)^{\frac{n-2}{2-s}}+c_{n,s}\left(\frac{r_{k}^{\frac{2-s}{2}}}{r_{k}^{2-s}+|x|^{2-s}}\right)^{\frac{n-2}{2-s}}\right),
\end{equation*}
where $\lim_{k\to +\infty}\eps_k=0$ uniformly on $B_{R r_k}(0)\setminus B_{R^{-1}r_{k+1}}(0)$.

\medskip\noindent When $q=\crit-1$, the full nonlinearity is conformally invariant, and the situation is somehow different. Indeed, essentially, singular solutions develop only a (CGS) type profile. This is the object of the second theorem:

\begin{thm}\label{thm:2} Let $u\in C^\infty(\ballp)$ be a positive solution to \eqref{eq:u}. We assume that $q=\crit-1$. Then there exists $\mu_0(n,s)>0$ such that:

\smallskip\noindent$\bullet$ If $\mu>\mu_0(n,s)$, then $0$ is a removable singularity,\par
\smallskip\noindent$\bullet$ If $\mu=\mu_0(n,s)$, then 
\begin{enumerate}
\item Either $0$ is a removable singularity,
\item Or $\lim_{x\to 0}|x|^{\frac{n-2}{2}}u(x)=\left(\frac{2-s}{2\mu_0(n,s)}\right)^{\frac{n-2}{2s}}$.
\end{enumerate}
\smallskip\noindent$\bullet$ If $\mu<\mu_0(n,s)$, then 
\begin{enumerate}
\item Either $0$ is a removable singularity,
\item Or there exists $c_1,c_2>0$ such that 
\begin{equation*}
c_1|x|^{-\frac{n-2}{2}}\leq u(x)\leq c_2|x|^{-\frac{n-2}{2}}\hbox{ on }B_{1/2}(0)\setminus \{0\}.
\end{equation*}
\end{enumerate}
\end{thm}
\noindent The explicit value of $\mu_0(n,s)$ is
\begin{equation*}
\mu_0(n,s):=\frac{(2-s)s^{\frac{s}{2-s}}}{2^{\frac{2(1-s)}{2-s}}(n-2)^{\frac{2s}{2-s}}}.
\end{equation*}
The proof of Theorems \ref{thm:1} and \ref{thm:2} rely on various pointwise estimates and the use of Pohozaev-type identities. Indeed, our first task is to provide a pointwise control of $W$ in Section \ref{sec:pointwise:est}: this will enable us to show that either the profile is of type (ND) or $W$ is controled from above by a constant. When $\mu>0$ and $q\neq\crit-1$, the classical Pohozaev integral on a ball is not constant (see the definition in \eqref{def:P:invar}), but it has a limit (the asymptotic Pohozaev integral) when the radius of the ball goes to $0$. The value of the asymptotic Pohozaev integral differentiates the two profiles (CGS) and (MB) (respectively when it is positive or null). Here, it is to be noticed that the nonconstant Pohozaev integral generates the Multi-Bump profile (MB): in the conformally invariant equation $-\Delta U=|x|^{-s} U^{\crits-1}$, the Pohozaev integral is constant, which imposes a positive lower-bound for \eqref{key} and then a (CGS) profile (see Caffarelli-Gidas-Spruck \cite{cgs} or Korevaar-Mazzeo-Pacard-Schoen \cite{kmps}, see also Marques \cite{marques} for the case of a non-Euclidean metric). When there is no positive lower bound, the situation is more intricate and we perform a blow-up analysis in the spirit of Druet-Hebey-Robert \cite{dhr} to obtain the (MB) profile.

\medskip\noindent The article \cite{kmps} of Korevaar-Mazzeo-Pacard-Schoen was an important source of inspiration of this work. Concerning bibliographic references, apart from the articles already mentioned in the introduction, there is a huge litterature about the case of a sole convex nonlinear problem, that is for $-\Delta u=-\mu u^q$, with interior or boundary singularity: we refer to the classical monograph by V\'eron \cite{veron} and the more recent contribution \cite{PV} by Porretta-V\'eron. We also refer to the monograph \cite{cirstea:book} by the first author for an exhaustive study of such problems.

\medskip\noindent The paper is organized as follows. Section \ref{sec:pointwise:est} is devoted to the proof of a general pointwise estimate for solutions to \eqref{eq:u}. Some consequences of this estimate are provided in Sections \ref{sec:remove}, \ref{sec:compensate} and \ref{sec:aux}. In Section \ref{sec:poho}, we make a full study of solutions to the limiting equation \eqref{eq:U:rnp} on $\rnp$ and introduce the Pohozaev integral. The optimal control of solutions is proved in Sections \ref{sec:blowup:1} and \ref{sec:blowup:2} for the (MB) profile. This is used in Section \ref{sec:radii} to estimate the rescaling parameters associated to solutions to \eqref{eq:u}. Theorem \ref{thm:1} is proved in Section \ref{thm:1}. Section \ref{sec:q:crit} is devoted to the specific case $q=\crit-1$ and the proof of Theorem \ref{thm:2}. 

\medskip\noindent{\bf Notation.} In all the paper, $C$ will denote a generic positive constant, the value of which might change from one line to the other, potentially even in the same line. We denote $\omega_{n-1}:=|\partial B_1(0)|$ the volume of $\mathbb{S}^{n-1}$, the Euclidean unit $(n-1)-$sphere. 

\section{A first pointwise estimate for $u$}\label{sec:pointwise:est}

The aim of this section is to obtain upper bound estimates for any positive solution 
$u\in C^\infty(B_1(0)\setminus\{0\})$ of \eqref{eq:u}, namely $\limsup_{x\to 0} |x|^pu(x)<\infty$, where $p>0$ is given by \eqref{def:p}. We refer to Proposition~\ref{prop:bnd:u}, whose proof in \S\ref{apb} uses a contradiction argument and relies essentially on Lemma~\ref{lem:blowup} to be introduced shortly in \S\ref{sec:lem}.  
Before presenting the details, let us 
summarize several important facts in connection with $\Lambda:=\limsup_{x\to 0} |x|^{\frac{n-2}{2}}u(x)$ to be proved in Sections \ref{sec:pointwise:est}, \ref{sec:remove}, \ref{sec:compensate}:
\begin{enumerate}
\item  $\Lambda<\infty$
for every $1<q\leq \crits-1$ and $q=\crit-1$ (see Proposition~\ref{prop:bnd:u});
\item If $\crits-1<q<\crit-1$, then 
\begin{itemize}
\item $\Lambda<\infty$ if and only if $\lim_{x\to 0}|x|^s u(x)^{q-(\crits-1)}=0$ (by Corollary~\ref{co1});
\item If $\Lambda=\infty$, then $\lim_{x\to 0}|x|^s u(x)^{q-(\crits-1)}=\mu^{-1}$ (see Proposition~\ref{prop:lim:1}).
\end{itemize}
\item If $q>\crit-1$, then $\Lambda=0$. Moreover, zero is a removable singularity for $u$ provided that $\Lambda=0$ (see Proposition~\ref{prop:remove} and Corollary~\ref{coro:remove}).
\end{enumerate}

\subsection{A general Lemma}\label{sec:lem}

Let $u\in C^\infty(B_1(0)\setminus\{0\})$ be a positive solution to \eqref{eq:u}. For $\epsilon>0$, we define
\[  d_\eps(x):=\max\{|x|-\eps, 0\}\hbox{ for all }x\in\mathbb{R}^n.\]
Let $a>0$ and $b\in\rr$ be fixed. 
For $\eps\in (0,1/2)$, we define $\we$ on $B_1(0)$ as follows 
\begin{equation} \label{a0} \left\{ \begin{aligned}
&\we(x):=d_\eps(x)^a|x|^b u(x)\quad \mbox{for every } x\in B_1(0)\setminus\{0\},\\
& \we(0)=0.
\end{aligned} \right.
\end{equation} 
Since $\we\in C(\overline{B_{1/2}(0)})$, we see that there exists $\xe\in \overline{B_{1/2}(0)}\setminus\{0\}$ such that
\begin{equation} \label{a1} \max_{x\in \overline{B_{1/2}(0)}}\we(x)=\we(\xe)>0.\end{equation}
Up to a subsequence, we assume that 
\begin{equation}\tag{H1} \label{H1} \lim_{\eps\to 0}\we(\xe)=+\infty, \end{equation}
then since $u$ is smooth on $B_1(0)\setminus\{0\}$, we infer that there exists $\eps_1>0$ such that 
\begin{equation} \label{aa} \lim_{\epsilon\to 0}|\xe|=0\quad \mbox{and }\  d_\eps(\xe)=|\xe|-\eps>0\ \ \mbox{for every }\eps\in (0,\eps_1). \end{equation}
Our next result is essential for proving the {\em a priori} estimates in Proposition~\ref{prop:bnd:u}. 

\begin{lem}\label{lem:blowup} Let $a>0$ and $b\in \rr$. For $0<\eps<1/2$, we define $\we$ and $\xe$ as in \eqref{a0} and \eqref{a1}, respectively. 
Suppose that, up to a subsequence,  \eqref{H1} holds, $\lim_{\eps\to 0}u(\xe)=+\infty$ and for a family 
$(\lae)_\eps$ of positive numbers converging to zero as $\eps\to 0$, we have 
\begin{equation} \tag{H2} \label{H2} \lim_{\eps\to 0}\frac{|\xe|-\eps}{\lae}=+\infty. \end{equation}
We assume that there exist non-negative numbers $\alpha$ and $ \beta$ such that
\begin{equation} \label{H3} \tag{H3} 
\lim_{\eps\to 0}\frac{\lae^2u(\xe)^{\crits-2}}{|\xe|^s}=\alpha\hbox{ and }\lim_{\eps\to 0}\lae^2 u(\xe)^{q-1}=\beta.
\end{equation}
Then there exists $U\in C^\infty(\rn)$ such that
\begin{equation} \label{def:u} \left\{\begin{aligned}
-&\Delta U=\alpha U^{\crits-1}-\beta \mu U^q&&\hbox{ in }\rn,&\\
&0<U(x)\leq U(0)=1&&\hbox{ for all }x\in\rn.&
\end{aligned}\right.\end{equation}
\end{lem}

\smallskip\noindent{\it Proof of Lemma \ref{lem:blowup}:} We define a family of functions $\ue$ as follows
\begin{equation} \label{a3} 
\ue(x):=\frac{u(\xe+\lae x)}{u(\xe)}\quad \mbox{ for all }x\in \frac{(B_1(0)\setminus\{0\})-\xe}{\lae}. \end{equation}

\medskip\noindent We claim that for every $R>0$ and every $\eta\in (0,1)$, there exists 
$\epsilon(R,\eta)>0$ such for any $\eps\in (0,\eps(R,\eta))$, we can define $\ue$ in $B_R(0)$ and  
\begin{equation}\label{bnd:ue:C}
0<\ue(x)\leq (1+\eta)^{a+|b|}\quad \mbox{for all }x\in B_R(0).
\end{equation}
We prove the claim. For every $x\in B_R(0)$ and every $\eps>0$, we have
\[ \left\{
\begin{aligned}
&  1-\frac{\lambda_\eps}{|x_\eps|} R\leq \frac{|x_\eps+\lambda_\eps x|}{|x_\eps|} \leq 1+\frac{\lambda_\eps}{|x_\eps|} R,\\
& 1-\frac{\lambda_\eps}{|x_\eps|-\eps} R\leq \frac{|x_\eps+\lambda_\eps x|-\eps}{|x_\eps|-\eps}\leq 1+\frac{\lambda_\eps}{|x_\eps|-\eps}R. 
\end{aligned} \right. 
\]
From \eqref{H2}, we find that $\lim_{\eps\to 0} \lambda_\eps/|x_\eps|=\lim_{|\eps|\to 0} \lambda_\eps/(|x_\eps|-\eps)=0$. 
Hence, for every $\eta>0$, there exists $\eps(R,\eta)\in (0,\eps_1)$ such that
\begin{equation}\label{bnd:d:1}
\frac{1}{1+\eta}\leq \frac{|\xe+\lae x|-\eps}{|\xe|-\eps}\leq 1+\eta,\quad \frac{1}{1+\eta}\leq \frac{|\xe+\lae x|}{|\xe|}\leq 1+\eta 
\end{equation}
for all $x\in B_R(0)$ and all $\eps\in (0,\eps(R,\eta))$. 
Therefore, $\xe+\lae x\in B_{1/2}(0)$ and $\ue(x)$ is well defined, so that $\we(\xe+\lae x)\leq \we(\xe)$. This yields  
$$d_\eps(\xe+\lae x)^a|\xe+\lae x|^b u(\xe+\lae x)\leq d_\eps(\xe)^a|\xe|^b u(\xe)$$
for all $x\in B_R(0)$ and $\epsilon\in (0,\epsilon(R,\eta))$. Then, by \eqref{aa}, \eqref{a3} and \eqref{bnd:d:1}, we get \eqref{bnd:ue:C}. This proves the claim.

\medskip\noindent We fix $R>0$ and $\eta\in (0,1)$. It follows from \eqref{bnd:ue:C} that 
$\ue$ is uniformly bounded on $B_R(0)$ with respect to $\eps>0$ small enough.
Since $u$ is a positive solution of \eqref{eq:u}, we see that $\ue$ satisfies
\begin{equation}\label{eq:ue}
-\Delta\ue=\frac{\lae^2u(\xe)^{\crits-2}}{|\xe|^s}\frac{\ue^{\crits-1}}{\left|\frac{\xe}{|\xe|}+\frac{\lae }{|\xe|}x\right|^s}-\lae^2 u(\xe)^{q-1}\mu \ue^{q}\quad \mbox{in } B_R(0). 
\end{equation}
Thus using \eqref{H3}, \eqref{bnd:ue:C} and standard elliptic theory (see, for instance,
Gilbarg--Trudinger \cite{gt}), we conclude that there exists $U\in C^2(\rn)$ such that 
\[ \ue\to U\ \ \mbox{in } C^2_{\rm loc}(\rn)\ \ \mbox{as }\eps\to 0,\] where $U$ is a non-negative solution of 
\[ -\Delta U=\alpha U^{\crits-1}-\beta \mu U^q \quad \mbox{ in }\rn. \]
Moreover, letting $\eps\to 0$ and then $\eta\to 0$ in \eqref{bnd:ue:C}, we find that $0\leq U(x)\leq 1$ for all $x\in\rn$. Since $U(0)=\lim_{\eps\to 0}\ue(0)=1$, it follows from Hopf's maximum principle that $U>0$ in $\rn$. Therefore $U$ satisfies \eqref{def:u}. This ends the proof of Lemma \ref{lem:blowup}.\hfill$\Box$

\subsection{A priori bounds} \label{apb}
For convenience, we define $p$ as follows:
\begin{equation} \label{def:p} p:=\left\{\begin{aligned}
& \frac{n-2}{2} && \hbox{ if }1<q\leq \crits-1,&\\
& \frac{s}{q-(\crits-1)} && \hbox{ if }\crits-1<q<\crit-1,&\\
&\frac{2}{q-1} && \hbox{ if }q\geq \crit-1.&
\end{aligned}\right.\end{equation}
This subsection is essentially devoted to the proof of the following result.
\begin{prop}\label{prop:bnd:u} Let $u\in C^\infty(B_1(0)\setminus\{0\})$ be a positive solution of \eqref{eq:u} and let $p$ be given by \eqref{def:p}. Then there exists a positive constant 
$C$ such that 
\begin{equation}\label{bnd:u}
u(x)\leq C |x|^{-p} \quad \mbox{ for all }x\in B_{1/2}(0)\setminus \{0\}. 
\end{equation}
\end{prop}

\begin{proof}[Proof of Proposition \ref{prop:bnd:u}]  We take inspiration in Korevaar-Mazzeo-Pacard-Schoen \cite{kmps} where a similar upper bound was proved for $s=\mu=0$. We take $a>0$ and $b\in \rr$ be such that
\begin{equation}\label{choice:a:b:bis}
 \left\{\begin{array}{ll}
\displaystyle a:=\frac{2}{\crits-2}\hbox{ and }b:=-\frac{s}{\crits-2}& \hbox{ if }1<q\leq\crits-1,\\
\displaystyle a:=p  \ \ \mbox{and } b:=0&\hbox{ if }q>\crits-1.\\
\end{array}\right.
\end{equation}
Notice that $a>0$ and $a+b=p>0$, where $p$ is defined as in \eqref{def:p}. 
For any $\eps\in (0,1/2)$, we define $\we$ as in \eqref{a0} with $a$ and $b$ as above. 
To prove \eqref{bnd:u}, our objective is to bound $\we$ uniformly. We argue by contradiction. Let $\xe$ be given by \eqref{a1}, and assume that \eqref{H1} holds, that is $\lim_{\eps\to 0}\we(\xe)=+\infty$. Using \eqref{aa}, we find that $\we(\xe)\leq |\xe|^{a+b} u(\xe)$ and thus $\lim_{\eps\to 0} u(\xe)=+\infty$. This implies that $\lae\to 0$ as $\eps\to 0$, where 
for every $\eps\in (0,1/2)$, we define $\lambda_\eps$ as follows
\begin{equation} \label{lambda-eps}
\lambda_\eps=\left\{ 
\begin{array}{ll}
|x_\eps|^{\frac{s}{2}} u(x_\eps)^{-\frac{\crits-2}{2}}& \hbox{ if }1<q\leq\crits-1,\\
 u(x_\eps)^{-\frac{q-1}{2}}&\hbox{ if } q>\crits-1.
\end{array}\right.
\end{equation}

\noindent{\bf Case 1: $q> \crits-1$}. 
By \eqref{a0} and \eqref{H1}, we have $\lim_{\eps\to 0}(|\xe|-\eps)^{p}u(\xe)=+\infty$ since $a=p$ and $b=0$. 
But $p=\max\left\{\frac{2}{q-1},\frac{s}{q-(\crits-1)}\right\}$ if $q>\crits-1$ so that
$$ \lim_{\eps\to 0} (|\xe|-\eps)^{\frac{2}{q-1}}u(\xe)=+\infty\quad \text{and} \quad \lim_{\eps\to 0} (|\xe|-\eps)^{\frac{s}{q-(\crits-1)}}u(\xe)=+\infty. $$
Thus, \eqref{H2} holds for $\lae$ defined by \eqref{lambda-eps}, whereas \eqref{H3} holds for $\alpha=0$ and $\beta=1$. 
By applying Lemma \ref{lem:blowup}, we conclude that there exists $U\in C^\infty(\rn)$ such that 
\[ \left\{  \begin{aligned}
 -& \Delta U=-\mu U^q && \mbox{in } \rn&&\\
& 0<U(x)\leq U(0)=1 && \mbox{for all } x\in \rn.
\end{aligned}\right. \] Since $0$ is a point of maximum for $U$, we have that $-\Delta U(0)\geq 0$, which is a contradiction. Thus, \eqref{H1} cannot hold in Case 1. 

\smallskip
\noindent{\bf Case 2: $1<q\leq\crits-1$}. 
From \eqref{H1} and \eqref{aa}, jointly with \eqref{choice:a:b:bis}, we find that 
$$ \lim_{\eps\to 0}(|\xe|-\eps) |\xe|^{-\frac{s}{2}} \,[u(\xe)]^{\frac{\crits-2}{2}}=+\infty,$$ which proves \eqref{H2} for $\lae$ given by \eqref{lambda-eps}. 
Moreover, $\lambda_\eps$ satisfies
\begin{equation} \label{new1} \frac{\lae^2 u(\xe)^{\crits-2}}{|\xe|^s}=1\hbox{ and }\lae^2 u(\xe)^{q-1}=|\xe|^s u(\xe)^{q-(\crits-1)}.\end{equation}
Since $q\leq \crits-1$ and $\lim_{\eps\to 0} u(\xe)=+\infty$, from \eqref{new1}, we see that \eqref{H3} holds with $\alpha=1$ and $\beta=0$.
Thus, by Lemma \ref{lem:blowup}, there exists $U\in C^\infty(\rn)$ such that 
\[ \left\{ \begin{aligned}
-&\Delta U=U^{\crits-1}&& \mbox{in } \rn, &\\ 
& 0<U(x)\leq U(0)=1 && \mbox{for every } x\in \rn.
\end{aligned} \right.\]
 But this is impossible from Caffarelli--Gidas--Spruck \cite{cgs} since $\crits<\crit$ (we use that $s>0$). Then \eqref{H1} does not hold, which ends Case 2.

\medskip\noindent Hence, in both cases, there exists $C>0$ such that $\we(x)\leq C$ for all $x\in B_{1/2}(0)\setminus\{0\}$ and $\eps\in (0,1/2)$. Letting $\eps\to 0$ yields 
Proposition \ref{prop:bnd:u}.\end{proof}

\begin{coro} \label{co1} If $u\in C^\infty(B_1(0)\setminus\{0\})$ is a positive solution of \eqref{eq:u} such that 
\neweq{oo} \lim_{|x|\to 0} |x|^s u(x)^{q-(\crits-1)}=0,
\endeq
then there exists a positive constant $C$ such that 
\neweq{oo1}
u(x)\leq C |x|^{-\frac{n-2}{2}} \quad \mbox{in } B_{1/2}(0)\setminus\{0\}. 
\endeq 
\end{coro} 

\smallskip\noindent{\it Proof of Corollary \ref{co1}.} We proceed as in Case 2 in the proof of Proposition~\ref{prop:bnd:u}. 
The only change is that $\beta=0$ in \eqref{H3} follows here from \eqref{new1} and \eqref{oo}. \hfill$\Box$

\section{Removable singularities}\label{sec:remove}

\begin{prop}\label{prop:remove} Let $u\in C^\infty(B_1(0)\setminus\{0\})$ be a positive solution to \eqref{eq:u} such that 
\begin{equation}\label{lim:u:0}
\lim_{x\to 0}|x|^{\frac{n-2}{2}}u(x)=0.
\end{equation}
Then $0$ is a removable singularity for $u$.
\end{prop}

\medskip\noindent{\it Proof of Proposition \ref{prop:remove}:} We define the operator 
$$L\phi:=-\Delta\phi-\frac{u^{\crits-2}}{|x|^s}\phi+\mu u^{q-1}\phi$$
for $\phi\in C^2(B_1(0)\setminus\{0\})$. We fix $\alpha\in (0,n-2)$. Direct computations yield
$$L(|x|^{-\alpha})=|x|^{-\alpha-2}\left(\alpha(n-2-\alpha)-\left(|x|^{\frac{n-2}{2}}u(x)\right)^{\crits-2}+\mu|x|^2u(x)^{q-1}\right)$$
on $B_1(0)\setminus\{0\}$. Since $\alpha\in (0,n-2)$ and \eqref{lim:u:0} holds, there exists $R(\alpha)>0$ such that
\begin{equation}\label{ineq:L:alpha}
L(|x|^{-\alpha})>0
\end{equation}
for all $x\in\rn$ such that $0<|x|<R(\alpha)$. We fix $\beta\in (0,n-2)$, and we let $0<r<\delta<\min\{R(\alpha),R(\beta)\}$ be two real numbers, and we define the function
$${\mathcal H}(x):=\left(\sup_{|z|=r}|z|^\alpha u(z)\right)|x|^{-\alpha}+\left(\sup_{|z|=\delta}|z|^\beta u(z)\right)|x|^{-\beta }\hbox{ for all }x\neq 0.$$
It follows from \eqref{ineq:L:alpha} and the definition of $\mathcal H$ that
$$\left\{\begin{array}{ll}
L({\mathcal H}-u)>0 & \hbox{ in }B_\delta(0)\setminus \bar{B}_r(0)\\
({\mathcal H}-u)(x)\geq 0 & \hbox{ for all }x\in \partial\left(B_\delta(0)\setminus \bar{B}_r(0)\right)
\end{array}\right\}$$ 
Since $L{\mathcal H}>0$ et ${\mathcal H}>0$ sur $\overline{B_\delta(0)\setminus \bar{B}_r(0)}$, it follows from Beresticky-Nirenberg-Varadhan \cite{bnv} that $L$ satisfie the comparison principle and therefore ${\mathcal H}\geq u$ on $B_\delta(0)\setminus \bar{B}_r(0)$. Taking $\alpha>\frac{n-2}{2}$ and $\delta>0$ small enough, using \eqref{lim:u:0} and letting $r\to 0$ yields 
$$u(x)\leq \left(\sup_{|z|=\delta}|z|^\beta u(z)\right)|x|^{-\beta }\hbox{ for all }0<|x|<\delta.$$

\medskip\noindent Therefore, since $s\in (0,2)$, we get that there exists $p>\frac{n}{2}$ such that $|x|^{-s}u^{\crits-1}-\mu u^q\in L^p(B_{1/2}(0))$. It then follows from Theorem 1 of Serrin \cite{serrin} and \eqref{lim:u:0} that the singularity at zero is removable. This ends the proof of Proposition \ref{prop:remove}.\hfill$\Box$

\begin{coro}\label{coro:remove} Let $q>\crit-1$ and $u\in C^\infty(B_1(0)\setminus\{0\})$ be a positive solution of \eqref{eq:u}. Then $0$ is a removable singularity.
\end{coro}
\begin{proof}[Proof of Corollary~\ref{coro:remove}] Proposition \ref{prop:bnd:u} gives that $|x|^\frac{2}{q-1}u(x)\leq C$ on $B_{1/2}(0)\setminus\{0\}$ for some constant $C>0$. Since $q>\crit-1$, this yields $\lim_{x\to 0}|x|^{\frac{n-2}{2}}u(x)=0$. Using Proposition \ref{prop:remove}, we complete the proof of Corollary \ref{coro:remove}. \end{proof}

\section{The case $\crits-1<q<\crit-1$: preliminary analysis}\label{sec:compensate}

Our aim in this section is to establish the following result. 

\begin{prop}\label{prop:lim:1} Let $\crits-1<q<\crit-1$. If $u\in C^\infty(B_1(0)\setminus\{0\})$ is a positive solution 
to \eqref{eq:u}, then the following dichotomy holds:
\begin{equation}\label{dicho}
\hbox{either }\lim_{x\to 0}|x|^\frac{s}{q-(\crits-1)}u(x)= 0\hbox{ or }\lim_{x\to 0}|x|^\frac{s}{q-(\crits-1)}u(x)= \mu^{-\frac{1}{q-(\crits-1)}}.
\end{equation}
Moreover, in the first case, we have that $\limsup_{x\to 0} |x|^{\frac{n-2}{2}}u(x)<+\infty$.
\end{prop}
\begin{proof}[Proof of Proposition \ref{prop:lim:1}] We consider a sequence $(x_i)_{i\geq 1}$ such that 
$|x_i|\in (0,1/4)$ for all $i\geq 1$ and 
$\lim_{i\to +\infty}x_i=0$. Suppose that 
\begin{equation}\label{lim:l}
\lim_{i\to +\infty}|x_i|^\frac{s}{q-(\crits-1)}u(x_i)=\ell\quad \text{for some } \ell\in (0,\infty).
\end{equation}
We claim that $\ell=\mu^{-\frac{1}{q-(\crits-1)}}$.

\smallskip\noindent We prove the claim. Clearly, $\lim_{i\to +\infty}u(x_i)=+\infty$. It follows that
\begin{equation}\label{lim:1:infty}
\lim_{i\to +\infty}|x_i|^{\frac{n-2}{2}}u(x_i)=+\infty.
\end{equation}
We define $\lambda_i$ as a sequence of positive numbers converging to $0$, namely
$$\lambda_i:=|x_i|^{\frac{s}{2}}u(x_i)^{-\frac{2-s}{n-2}}\quad \text{for all } i\geq 1.$$
Using \eqref{lim:1:infty}, we find that
$\lim_{i\to +\infty}\frac{\lambda_i}{|x_i|}=0$.
Therefore, for all $R>0$, there exists $i_R\geq 1$ large such that 
\begin{equation}\label{bnd:x:lambda}
|x_i|/2\leq |x_i+\lambda_i x|\leq 2|x_i|<1/2\ \ \text{for all }x\in B_R(0)\ \text{and } i\geq i_R. 
\end{equation}
Consequently, $u_i(x)$ is well-defined on $B_R(0)$ for all $i\geq i_R$, where we set
$$u_i(x):=\frac{u(x_i+\lambda_i x)}{u(x_i)}\quad \text{for all } x\in B_{\frac{|x_i|}{2\lambda_i}}(0)\ \text{and } i\geq 1.$$
By Proposition \ref{prop:bnd:u}, there exists $C>0$ such that
$|x|^{\frac{s}{q-(\crits-1)}} u(x)\leq C$ for all $x\in B_{1/2}(0)\setminus\{0\}$. Let $C_1:=2^{\frac{s}{q-(\crits-1)}}C$. Therefore, \eqref{lim:l}, \eqref{lim:1:infty} and 
\eqref{bnd:x:lambda} yield
\begin{equation}\label{bnd:ui:ter}
u_i(x)\leq \frac{C |x_i+\lambda_i x|^{-\frac{s}{q-(\crits-1)}}}{u(x_i)}\leq  \frac{C_1}{|x_i|^{\frac{s}{q-(\crits-1)}}u(x_i)} 
\leq \frac{2C_1}{\ell}
\end{equation}
for $x\in B_R(0)$ provided $i$ is large enough. Equation \eqref{eq:u} rewrites as 
$$-\Delta u_i=\frac{u_i^{\crits-1}}{\left|\frac{x_i}{|x_i|}+ \frac{\lambda_i x}{|x_i|}\right|^s}- |x_i|^su(x_i)^{q-(\crits-1)} \mu u_i^q
\quad \text{in } B_{\frac{|x_i|}{2\lambda_i}}(0). $$
Hence, by \eqref{bnd:ui:ter} and the standard elliptic theory (see, for instance, Gilbarg--Trudinger \cite{gt}), 
there exists $U\in C^2(\rn)$ so that up to a subsequence, $\lim_{i\to +\infty}u_i=U$ in $C^2_{\rm loc}(\rn)$. Moreover, $U(0)=1$ and 
$U$ is a non-negative bounded solution of 
$$ 
-\Delta U=U^{\crits-1}- \ell^{q-(\crits-1)} \mu U^q\ \ \text{in }\rn.$$
By Hopf's maximum principle, we have $U>0$ in $\rn$. From Lemma \ref{lem:hypothesis} below, we conclude that 
$U$ is constant, and thus $\ell^{q-(\crits-1)}\mu=1$. This proves the claim.

\medskip\noindent Hence, given any sequence $(z_i)_i\to 0$, then, up to a subsequence, 
$|z_i|^\frac{s}{q-(\crits-1)}u(z_i)$ converges to either $0$ or $\mu^{-\frac{1}{q-(\crits-1)}}$ (by using Proposition \ref{prop:bnd:u}). 
By a continuity argument, we get \eqref{dicho}. In the first case, the control on $u$ follows from Corollary~\ref{co1}. 
This proves Proposition~\ref{prop:lim:1}. 
\end{proof}

\begin{lem}\label{lem:hypothesis} Let $1<r<q<\crit-1$. If $\alpha$ is a positive number and $U\in C^2(\rn)$ is a positive bounded 
solution of $-\Delta U=U^r-\alpha U^q$ in $\rn$, then $U\equiv \alpha^{\frac{1}{r-q}}$ in $\rn$. 
\end{lem}

\begin{proof}[Proof of Lemma \ref{lem:hypothesis}] We claim that $U(x)\leq \alpha^{1/(r-q)}$ for all $x\in\rn$. Indeed, if $U$ achieves its maximum (say $M$) at some point $x_0\in \rn$, then the claim follows from $-\Delta U(x_0)\geq 0$. If $U$ does not achieves its maximum $M$, then let $(x_i)_i\in\rn$ be such that $\lim_{i\to +\infty}U(x_i)=M$. Define $U_i:=U(\cdot+x_i)$. Then $U_i$ is bounded by $M$ and satisfies the same equation as $U$. It then follows from elliptic theory that $U_i\to \tilde{U}$ in 
$C^2_{\rm loc}(\rn)$, where $-\Delta \tilde{U}=\tilde{U}^r-\alpha \tilde{U}^q$ in $\rn$ and $\max_{\rn}\tilde{U}=\tilde{U}(0)=M$. 
From the first part, we conclude that $M\leq \alpha^{1/(r-q)}$. This proves the claim.

\smallskip\noindent We  define $g(t):=t^r-\alpha t^q$ for $0\leq t\leq \alpha^{\frac{1}{r-q}}$ and $g(t):=0$ otherwise. We see that $g(\alpha^{\frac{1}{r-q}})=0$, while 
$g$ is positive on $(0,\alpha^{\frac{1}{r-q}})$ and
$t^{-\frac{n+2}{n-2}}g(t)$ is decreasing on $(0, \alpha^{\frac{1}{r-q}}]$. Then by Theorem~3 in \cite{bi} (or, Theorem~1.3 in \cite{li}), we conclude that $U\equiv \alpha^{\frac{1}{r-q}}$ in $\rn$. This proves Lemma \ref{lem:hypothesis}.\end{proof}

\section{Auxiliary results for $q\leq \crit-1$}\label{sec:aux}

In this section, let $q\leq \crit-1$ and $u\in C^\infty(\ballp)$ be a positive solution to 
\begin{equation}\label{eq:5}
-\Delta u= \frac{u^{\crits-1}}{|x|^s}-\mu u^q \quad \hbox{in }\ballp.
\end{equation}
We assume that $\limsup_{x\to 0} |x|^{\frac{n-2}{2}}u(x)<+\infty $. Hence, there 
exists $C>0$ such that 
\begin{equation}\label{control}
|x|^{\frac{n-2}{2}}u(x)\leq C\hbox{ for all }x\in B_{1/2}(0)\setminus\{0\}.
\end{equation}

Let $(t_i)_{i\geq 1}$ be a sequence of positive numbers with $\lim_{i\to +\infty}t_i=0$. We define 
\neweq{ui} u_i(x):=t_i^{\frac{n-2}{2}}u(t_i x)\quad \text{for all } x\in B_{t_i^{-1}}(0)\setminus \{0\}.
\endeq
Then equation \eqref{eq:5} rewrites for $u_i$ as follows
\begin{equation}\label{eq:ui}
-\Delta u_i=\frac{u_i^{\crits-1}}{|x|^s}-\mu t_i^{\frac{n-2}{2}(\crit-1-q)}u_i^q\quad \text{in } B_{t_i^{-1}}(0)\setminus \{0\}. 
\end{equation}

From \eqref{control}, we have $|x|^{\frac{n-2}{2}}u_i(x)\leq C$ for all $x\in B_{t_i^{-1}/2}(0)\setminus \{0\}$. By the standard elliptic theory, up to a subsequence, 
 $u_i\to U$ in $C^2_{\rm loc}(\rnp)$ as $i\to +\infty$, where $U\in C^2(\rnp)$ is a 
non-negative function. Passing to the limit in \eqref{eq:ui} yields
\begin{equation}\label{eq:lim:2}
-\Delta U=\frac{U^{\crits-1}}{|x|^s}-\lambda U^q\hbox{ in }\rnp\ \text{with } \lambda=\left\{
\begin{aligned}
& 0 && \text{ if }q<\crit-1,&\\
& \mu && \text{ if }q=\crit-1.&
\end{aligned}\right.
\end{equation}

Summarising, we have obtained the following result.

\begin{lem}\label{lem:prelim:1} Let $q\leq \crit-1$ and  
$u\in C^\infty(\ballp)$ be a positive solution of \eqref{eq:5}. Let $(t_i)_{i\geq 1}$ be a sequence of positive numbers
with $\lim_{i\to +\infty}t_i=0$. We define $(u_i)_i$ as in \eqref{ui}. If \eqref{control} holds, then $(u_i)_i$ satisfies, up to a subsequence, 
\neweq{jk} u_i\to U\  \text{in } C^2_{\rm loc}(\rnp)\ \text{as }i\to +\infty,\endeq where $U\in C^2(\rnp)$ is a non-negative solution of
\eqref{eq:lim:2}. 
 \end{lem}

In Lemma~\ref{lem:prelim:2} below, we shall rely on Lemma~\ref{lem:prelim:1} to obtain gradient and second derivative estimates on $u$.

\begin{lem}\label{lem:prelim:2} Let $q\leq\crit-1$ and $u\in C^\infty(\ballp)$ be a positive solution to \eqref{eq:50} such that \eqref{control} holds. 
Then there exists a positive constant $C>0$ such that
\begin{equation}\label{bnd:grad}
|x|^{\frac{n}{2}}|\nabla u|(x)+|x|^{\frac{n}{2}+1}|\nabla^2 u|(x)\leq C\hbox{ for all }x\in B_{1/2}(0)\setminus \{0\}.
\end{equation}
\end{lem}

\begin{proof}[Proof of Lemma \ref{lem:prelim:2}] We argue by contradiction. Assume that there exists $(x_i)_i\to 0$ such that $|x_i|^{\frac{n}{2}}|\nabla u(x_i)|+|x_i|^{\frac{n}{2}+1}|\nabla^2 u(x_i)|\to +\infty$ as $i\to +\infty$. We define $u_i(x):=|x_i|^{\frac{n-2}{2}}u(|x_i|x)$ for $x\in B_2(0)\setminus \{0\}$. From Lemma \ref{lem:prelim:1}, we have that, up to a subsequence, $(u_i)$ converges in $C^2_{\rm loc}(B_2(0)\setminus\{0\})$, and thus $|\nabla u_i(\frac{x_i}{|x_i|})|+|\nabla^2 u_i(\frac{x_i}{|x_i|})|$ is bounded as $i\to +\infty$, contradicting our initial hypothesis. This proves \eqref{bnd:grad}, which finishes the proof of Lemma~\ref{lem:prelim:2}. 
\end{proof}

We next establish a spherical Harnack inequality. 

\begin{lem} \label{lem:harnack}
Let $q\leq\crit-1$ and $u\in C^\infty(\ballp)$ be a positive solution to \eqref{eq:50} such that \eqref{control} holds. 
Then there exists a positive constant $C_1$ such that
\begin{equation} \label{harnack}
u(x)\leq C_1 u(y)\hbox{ for all }x,y\in \partial B_r(0)\ \text{and every }r\in (0,1/2).
\end{equation}
\end{lem}

\begin{proof}[Proof of Lemma~\ref{lem:harnack}] We argue by contradiction and assume that there exists a sequence $(t_i)_i\in (0, 1/2)$, and sequences $(x_i)_i,(y_i)_i\in\rn$ such that $|x_i|=|y_i|=t_i$ and $u(x_i)=o(u(y_i))$ as $i\to +\infty$. Without loss of generality, we assume that $t_i\to t\in [0,1/2]$ as $i\to +\infty$. We define $u_i$ as in \eqref{ui}. Then \eqref{eq:ui} holds. Moreover, from \eqref{control}, there exists $C>0$ such that $0<u_i(x)\leq C$ in $B_2(0)\setminus B_{1/2}(0)$. The classical Harnack inequality (see, for instance, Gilbarg--Trudinger \cite{gt}) gives a positive  constant $C_0$ such that $u_i(y)\leq C_0 u_i(x)$ for all $x,y\in \partial B_1(0)$. This contradicts $u(x_i)=o(u(y_i))$ as $i\to +\infty$, proving \eqref{harnack}. This ends the proof of Lemma~\ref{lem:harnack}. 
\end{proof}

\section{Pohozaev integral and first consequences}\label{sec:poho}

We prove the following result that will be used several times in the paper. 

\begin{prop}\label{prop:poho}
We fix a smooth bounded domain $\omega\subset \rn$ such that $0\not\in \overline{\omega}$. Let $v\in C^2(\overline{\omega})$ be any positive solution of 
\begin{equation}\label{eq:v:1}
-\Delta v=\frac{v^{\crits-1}}{|x|^s}-\lambda v^q\hbox{ in }\overline{\omega},
\end{equation}
where $\lambda\in\rr$, $s\in (0,2)$ and $q>1$. Then 
\begin{equation} \label{poho:id}
\begin{aligned}
&\int_{\partial \omega}\left[(x,\nu)\left(\frac{|\nabla v|^2}{2}-\frac{v^{\crits}}{\crits|x|^s}+\lambda\frac{v^{q+1}}{q+1}\right)-T(x,v)\partial_\nu v\right]\, d\sigma \\ 
&=\frac{(n-2)\lambda}{2(q+1)}(\crit-1-q)\int_\omega v^{q+1}\, dx, 
\end{aligned}
\end{equation}
where
\neweq{nott}
T(x,v):=(x,\nabla v(x))+\frac{n-2}{2}v(x).  
\endeq
Here, $(x,\nabla v(x)):=\sum_{j=1}^nx^j\partial_j v$, whereas $\nu$ and $d\sigma$ denote the outward normal vector of $\partial\omega$ and the canonical volume element on $\partial\omega$, respectively. 
\end{prop}

\begin{proof}[Proof of Proposition \ref{prop:poho}] The standard Pohozaev identity (see \cite{poho}) asserts that 
$$\int_\omega T(x,v)(-\Delta v)\, dx=\int_{\partial\omega}\left[(x,\nu)\frac{|\nabla v|^2}{2}-T(x,v)\,\partial_\nu v \right]\, d\sigma.$$
Independently, for any $\tau\in [0,2]$ and $p\geq 1$, integrating by parts yields
$$\int_\omega T(x,v) \frac{v^p}{|x|^\tau}\, dx=\left(\frac{n-2}{2}-\frac{n-\tau}{p+1}\right)\int_\omega\frac{v^{p+1}}{|x|^\tau}\, dx+\frac{1}{p+1}\int_{\partial\omega}\frac{(x,\nu)\,v^{p+1}}{|x|^\tau}\, d\sigma.$$
Combining these two identities with equation \eqref{eq:v:1} yields \eqref{poho:id}. This ends the proof of Proposition \ref{prop:poho}.\end{proof}

\subsection{Solutions of $-\Delta U= \frac{U^{\crits-1}}{|x|^s}$ in $\rn\setminus \{0\}$}\label{sec:lim:eq}
Throughout this section, we let
$U\in C^\infty(\rn\setminus\{0\})$ be a positive solution to
\begin{equation}\label{eq:lim}
-\Delta U= \frac{U^{\crits-1}}{|x|^s}\hbox{ in }\rnp. 
\end{equation}
For any $r>0$, we define the Pohozaev integral 
\begin{equation}\label{poho:invar:r}
P_r(U)=\int_{\partial B_r(0)}\left[(x,\nu)\left(\frac{|\nabla U|^2}{2}-\frac{U^{\crits}}{\crits|x|^s}\right)-T(x,U)\,\partial_\nu U\right]\, d\sigma.
\end{equation}
Letting $\lambda=0$ in Proposition \ref{prop:poho}, we find that 
$P_r(U)=P_1(U)$ for any $r>0$. For simplicity, we use $P(U)$ to denote this Pohozaev invariant associated to $U$.  

\vspace{0.2cm}
The following result, which follows essentially from Caffarelli--Gidas--Spruck \cite{cgs} and Hsia--Lin--Wang \cite{hlw}, shows that $0$ is a removable singularity for 
$U$ if and only if the Pohozaev invariant $P(U)$ is zero. In this case, there exists $\lambda>0$ such that $U$ is of the form \eqref{def:U:lambda}, where
$c_{n,s}$ is defined by 
\neweq{def:cns}
c_{n,s}=\left((n-s)(n-2)\right)^{\frac{n-2}{2(2-s)}}. 
\endeq

\begin{prop}\label{prop:poho:U} Let $U\in C^\infty(\rn\setminus\{0\})$ be a positive solution of \eqref{eq:lim}. Then $U$ is radially symmetrical with respect to $0$ and $P(U)\geq 0$. More precisely,

\smallskip\noindent$\bullet$ If $P(U)=0$, then $U$ extends continuously at $0$ and there exists $\lambda>0$ such that
\begin{equation}
\label{def:U:lambda}
U(x)=U_\lambda(x):=c_{n,s}\left(\frac{\lambda^{1-\frac{s}{2}}}{\lambda^{2-s}+|x|^{2-s}}\right)^{\frac{n-2}{2-s}}\hbox{ for all }x\in\rn. 
\end{equation}
\smallskip\noindent$\bullet$ If $P(U)>0$, then $U$ is singular at $0$ and there exists $v\in C^\infty(\rr)$ a positive periodic function such that
$$U(x)=|x|^{-\frac{n-2}{2}}v(-\ln|x|)\hbox{ for all }x\in \rnp.$$
Moreover, up to a translation, $v$ is uniquely defined by the value $P(U)>0$.
\end{prop}

\begin{proof}[Proof of Proposition \ref{prop:poho:U}]  We sketch the proof here for its steps will be used in the sequel. The radial symmetry has been proved by Chou--Chu \cite{cc} for removable singularity, and by Hsia--Lin--Wang \cite{hlw} for non-removable singularity. The methods are inspired by the classical moving-plane method of Alexandrov used by 
Caffarelli--Gidas--Spruck \cite{cgs} in the case $s=0$. We define 
\begin{equation}\label{def:phi}
\varphi :  \rr\times \mathbb{S}^{n-1}  \to  \rn\setminus \{0\}\ \text{by } 
\varphi(t,\theta) = e^{-t}\theta.
\end{equation}
The function $\varphi$ is a conformal diffeomorphism and $\varphi^*\hbox{Eucl}=e^{-2t}\left(dt^2+\hbox{can}_{n-1}\right)$, where $\hbox{can}_{n-1}$ is the canonical metric on $\mathbb{S}^{n-1}$. We write
$$U_\varphi(t,\theta):=e^{-\frac{n-2}{2}t}U(\varphi(t,\theta)).$$
By the invariance of the conformal Laplacian $L_g:=-\Delta_g+\frac{n-2}{4(n-1)}R_g$, we see that 
\begin{equation} \label{conf:lap}
\begin{aligned}
(-\Delta U)\circ \varphi(t,\theta)& = L_{\varphi^\star \hbox{Eucl}}(U\circ \varphi)
=e^{\frac{n+2}{2}t}L_{dt^2+\hbox{can}_{n-1}} e^{-\frac{n-2}{2}t}U\circ \varphi(t,\theta))\\
&=e^{\frac{n+2}{2}t}\left(-U_\varphi^{\prime\prime}(t)-\Delta_{\hbox{can}_{n-1}}U_\varphi+ \frac{(n-2)^2}{4}U_\varphi\right). 
\end{aligned} \end{equation}
Since $U$ is radially symmetrical, then $U_\varphi(t,\theta)$ is independent of $\theta$ and we define $v(t):=U_\varphi(t,\theta)$ for all $t\in\rr$ and $\theta\in \mathbb{S}^{n-1}$. It then follows from \eqref{conf:lap}, equation \eqref{eq:lim} and the definion of $v$ that
\begin{equation}\label{eq:v:2}
-v^{\prime\prime}+ \frac{(n-2)^2}{4}v= v^{\crits-1}\hbox{ in }\rr, \hbox{ with }v>0.
\end{equation}
Multiplying by $v^\prime$ and integrating, we get that there exists $K\in\rr$ such that
\begin{equation}\label{int:v}
-\frac{(v^\prime(t))^2}{2}+\frac{(n-2)^2}{8}v(t)^2-\frac{v(t)^{\crits}}{\crits}=K\hbox{ for all }t\in\rr.
\end{equation}
We define
$$K_{n,s}:=\frac{\crits-2}{2\cdot \crits}\left(\frac{(n-2)^2}{4}\right)^{\frac{\crits}{\crits-2}}.$$
A classical ODE analysis (see, for instance, Caffarelli--Gidas--Spruck \cite{cgs}) yields:

\medskip\noindent$\bullet$ Either $K=K_{n,s}$, and then $v\equiv \left(\frac{(n-2)^2}{4}\right)^{\frac{1}{\crits-2}}$ is constant,

\smallskip\noindent$\bullet$ Or $0<K<K_{n,s}$, and then there exists $T\in\rr$ such that $v\equiv v_K(\cdot-T)$, where $v_K$ is the unique nonconstant periodic solution to \eqref{eq:v:2} and \eqref{int:v} that achieves its minimum at $0$.

\smallskip\noindent$\bullet$ Or $K=0$ and then there exists $T\in\rr$ such that $v\equiv v_0(\cdot-T)$, where
$$v_0(t):=\left((n-s)(n-2)\right)^{\frac{1}{\crits-2}}\left(e^{\frac{2-s}{2}t}+e^{\frac{-(2-s)}{2}t}\right)^{-\frac{n-2}{2-s}}\hbox{ for all }t\in\rr.$$

\medskip\noindent In term of $U_\varphi$, the Pohozaev integral rewrites
$$\begin{aligned} 
P(U) =& \int_{\mathbb{S}^{n-1}}\left(-\frac{(U_\varphi^\prime(t,\theta))^2}{2}+\frac{(n-2)^2}{8}U_\varphi(t,\theta)^2-\frac{U_\varphi(t,\theta)^{\crits}}{\crits}\right)\, dv_{\hbox{can}}\\
& + \int_{\mathbb{S}^{n-1}} \frac{1}{2}|\nabla_\theta U_\varphi(t,\theta)|_{\hbox{can}}^2 
\, dv_{\hbox{can}}
\end{aligned}$$
for all $t\in\rr$. Since $v(t)=U_\varphi(t,\theta)$ for all $(t,\theta)\in\rr\times \mathbb{S}^{n-1}$, we get that
$$P(U)=\omega_{n-1}\left( -\frac{(v^\prime(t))^2}{2}+\frac{(n-2)^2}{8}v(t)^2-\frac{v(t)^{\crits}}{\crits}\right)$$
for all $t\in\rr$. Therefore, it follows from \eqref{int:v} that $P(U)=\omega_{n-1}K$.

\medskip\noindent The conclusion of Proposition \ref{prop:poho:U} then follows from the distinction above between the cases $K=0$ and $K>0$, and writing $U$ 
in terms of $v_K$, $K\geq 0$. \end{proof}

\subsection{The asymptotic Pohozaev integral for $q\leq \crit-1$} 
Let $u\in C^\infty(B_1(0)\setminus\{0\})$ be a positive solution of \eqref{eq:5}, namely $u$ satisfies
\begin{equation}\label{eq:00}
-\Delta u= \frac{u^{\crits-1}}{|x|^s}-\mu u^q \hbox{ in }\ballp.
\end{equation}
Recall that if $q>\crit-1$, then $0$ is a removable singularity for $u$. In this section, we assume that $q\leq \crit-1$.  For any $x\in \ballp$ and $t\geq 0$, we define 
\neweq{F-mu} 
\left\{ \begin{aligned}
& f_{\mu,q}(x,t):=\frac{t^{\crits-1}}{|x|^s}-\mu t^q \\
& F_{\mu,q}(x,t):=\int_0^t f_{\mu,q}(x,\xi)\,d\xi=\frac{t^{\crits}}{\crits|x|^s}-\mu\frac{t^{q+1}}{q+1}.
\end{aligned} \right.
\endeq
For any $r\in (0,1)$, we define the Pohozaev-type integral by
\neweq{def:P:invar} P_{r}^{(q)}(u):=\int_{\partial B_r(0)} \left[ (x,\nu) \left(\frac{|\nabla u|^2}{2}-F_{\mu,q}(x,u)\right)-T(x,u)\, \partial_\nu u\right]d\sigma,
\endeq where $T(x,u)$ is given by \eqref{nott} with $u$ instead of $v$.  By Proposition~\ref{prop:poho}, for every $0<r_1<r_2<1$, we have 
\neweq{id:poho:invar} P_{r_2}^{(q)}(u)-P_{r_1}^{(q)}(u)=
\frac{(n-2)}{2(q+1)}(\crit-1-q)\mu \int_{B_{r_2}(0)\setminus B_{r_1}(0)} u^{q+1}\, dx.
\endeq
\begin{prop}\label{prop:def:asymp} Let $u\in C^\infty(B_1(0)\setminus\{0\})$ be a positive solution of \eqref{eq:00}. We assume that $q\leq \crit-1$ and $\limsup_{x\to 0} |x|^{\frac{n-2}{2}} u(x)<\infty$. Then $(P^{(q)}_r(u))$ has a limit as $r\to 0$. We then define the {\it asymptotic Pohozaev integral} as 
\neweq{poh} P^{(q)}(u):=\lim_{r\to 0}P_r^{(q)}(u).\endeq
\end{prop}
\smallskip\noindent{\it Proof of Proposition \ref{prop:def:asymp}:} When $q=\crit-1$, $P^{(q)}_r(u)$ is independent of $r$ by \eqref{id:poho:invar} and the result is clear. We assume that $q<\crit-1$. It follows from \eqref{id:poho:invar} and $\limsup_{x\to 0} |x|^{\frac{n-2}{2}} u(x)<\infty$ that there exists $C>0$ such that for any $r_2\in (0,1)$ such that $r_1<r_2$, fixed, we have that 
$$0<P_{r_2}^{(q)}(u)-P_{r_1}^{(q)}(u)\leq C \int_{r_1}^{r_2}r^{n-(q+1)(n-2)/2-1}\, dr \leq C r_2^{\frac{n-2}{2}\left(\crit-1-q\right)}.$$
Therefore the limit of $P_{r}^{(q)}(u)$ as $r\to 0$ exists. This ends Proposition \ref{prop:def:asymp}.\qed

\section{Blow-up when $q<\crit-1$, Part I: first limiting profile}\label{sec:blowup:1}
Throughout this section, we let $u\in C^\infty(\ballp)$ be a positive solution to 
\begin{equation}\label{eq:50}
-\Delta u= \frac{u^{\crits-1}}{|x|^s}-\mu u^q \hbox{ in }\ballp
\end{equation}
such that there exists $C>0$ such that 
\begin{equation}\label{hyp:control}
|x|^{\frac{n-2}{2}}u(x)\leq C\hbox{ for all }x\in B_{1/2}(0)\setminus\{0\}
\end{equation}
and
\begin{equation}\label{hyp:inf:sup}
\liminf_{x\to 0}|x|^{\frac{n-2}{2}}u(x)=0\; ,\; \limsup_{x\to 0}|x|^{\frac{n-2}{2}}u(x)>0.
\end{equation}

\medskip\noindent The main point of this section is that, when $q<\crit-1$, \eqref{hyp:control} and \eqref{hyp:inf:sup} hold, then the limit $U$ obtained in Lemma \ref{lem:prelim:1} is either identically zero or a positive nonsingular regular solution to the limit equation 
\eqref{eq:lim} described by Proposition~\ref{prop:poho:U}. In particular, the singular solutions of \eqref{eq:lim} are ruled out. The case $q=\crit-1$ will be studied in detail in Section \ref{sec:q:crit}.
When $q>\crit-1$, then Corollary \ref{co1} gives that any solution to \eqref{eq:00} has a removable singularity, and then \eqref{hyp:inf:sup} cannot hold. 

\medskip\noindent We first prove that the limit obtained in Lemma \ref{lem:prelim:1} is a nonsingular solution to the limit equation \eqref{eq:lim}.

\begin{prop}\label{prop:prelim:1} Let $u\in C^\infty(\ballp)$ be a positive solution to \eqref{eq:50} such that $q<\crit-1$ and \eqref{hyp:control}, \eqref{hyp:inf:sup} hold. Then the asymptotic Pohozaev integral vanishes: $P^{(q)}(u)=0$. In particular, for any sequence $(t_i)\in (0,+\infty)$ such that $t_i\to 0$ as $i\to +\infty$, by defining $u_i(x):=t_i^{\frac{n-2}{2}}u(t_i x)$ for all $x\in B_{t_i^{-1}}(0)\setminus \{0\}$, then up to a subsequence, we have the following convergence in $C^2_{\rm loc}(\rnp)$ as $i\to +\infty$:
\begin{equation}\label{dicho:ui}
u_i\to \left\{\begin{array}{l}
\hbox{either }0\\
\hbox{or }U_\lambda:=c_{n,s}\left(\frac{\lambda^{1-\frac{s}{2}}}{\lambda^{2-s}+|\cdot|^{2-s}}\right)^{\frac{n-2}{2-s}}\hbox{ for some }\lambda>0
\end{array}\right\}.
\end{equation}
\end{prop}

\smallskip\noindent{\it Proof of Proposition \ref{prop:prelim:1}:} The convergence in \eqref{jk} to a nonnegative solution $U\in C^2(\rnp)$ of \eqref{eq:lim} is a consequence of Lemma \ref{lem:prelim:1}. With a change of variable, the Pohozaev integral  $P_{t_i}^{(q)}(u)$ (see \eqref{def:P:invar}) equals
$$\int_{\partial B_1(0)}\left[(x,\nu)\left(\frac{|\nabla u_i|^2}{2}-\frac{u_i^{\crits}}{\crits|x|^s}+\mu t_i^{\frac{n-2}{2}(\crit-1-q)} \frac{u_i^{q+1}}{q+1}\right)- T(x, u_i)\partial_\nu u_i\right]\, d\sigma$$
for all $i$. Therefore, letting $i\to +\infty$ and using the convergence of $P_r^{(q)}(u)$ to $P^{(q)}(u)$ as $r\to 0$ and of $(u_i)$ to $U$ as $i\to +\infty$, we get that
\begin{equation}\label{lim:P:u}
P^{(q)}(u)=\int_{\partial B_1(0)}\left[(x,\nu)\left(\frac{|\nabla U|^2}{2}-\frac{U^{\crits}}{\crits|x|^s}\right)- T(x,U)\partial_\nu U\right]\, d\sigma.
\end{equation}

\medskip\noindent We claim that 
\begin{equation}\label{invar:0}
P^{(q)}(u)=0.
\end{equation}
We prove the claim. It follows from \eqref{hyp:inf:sup} that there exists $(x_i)\in \ballp$ such that $x_i\to 0$ as $i\to +\infty$ and $|x_i|^{\frac{n-2}{2}}u(x_i)\to 0$ as $i\to +\infty$. We let $t_i:=|x_i|$ and define $u_i$ as above, and we let $\tilde{U}\in C^2(\rnp)$ be its limit in $C^2$. In particular, $u_i(\theta_i)=o(1)$ as $i\to +\infty$ where $\theta_i:=x_i/|x_i|\to \theta_\infty\in \partial B_1(0)$ as $i\to +\infty$. This yields $\tilde{U}(\theta_\infty)=0$. Since $\tilde{U}\geq 0$, it then follows from Hopf's strong comparison principle that $\tilde{U}\equiv 0$. Therefore, it follows from \eqref{lim:P:u} that \eqref{invar:0} holds. This proves the claim.

\medskip\noindent We claim that either $U\equiv 0$, or $U\equiv U_\lambda$ for some $\lambda>0$. We prove the claim. Since $U\geq 0$, it follows from Hopf's strong comparison principle that either $U\equiv 0$ or $U>0$. We assume that $U>0$. Then it follows from \eqref{lim:P:u} that $P(U)=P^{(q)}(u)$ (see also \eqref{poho:invar:r}), and then from the preceding claim, we get that $P(U)=0$. It then follows from Proposition \ref{prop:poho:U} that there exists $\lambda>0$ such that $U\equiv U_\lambda$. This proves the claim.

\medskip\noindent These claims prove Proposition \ref{prop:prelim:1}. \hfill$\Box$

\medskip\noindent For any $r\in (0,1)$, we define
\begin{equation}\label{def:w}
w(r):=r^{\frac{n-2}{2}}\bu(r),
\end{equation}
where, for any $f\in C^0(\ballp)$, we define 
$$\bar{f}(r):=\frac{1}{|\partial B_r(0)|}\int_{\partial B_r(0)}f\, d\sigma.$$ 
We now construct specific radii at which $u$ behaves nicely after rescaling.

\begin{prop}\label{prop:construct:radii} Let $u\in C^\infty(\ballp)$ be a positive solution to \eqref{eq:50} such that $q<\crit-1$ and \eqref{hyp:control}, \eqref{hyp:inf:sup} hold. Then there exist two sequences $(r_k)_k,(\tau_k)_k$ of positive numbers going to $0$ as $k\to +\infty$ such that for all $k\in\nn$,
\begin{equation}\label{mono:w:bis}
w'(r)<0\hbox{ for all }r\in (r_{k+1}, \tau_{k+1})\hbox{ and }w'(r)>0 \hbox{ for all }r\in (\tau_{k+1}, r_k).
\end{equation}
Moreover,
$$\lim_{k\to +\infty}r_k^{\frac{n-2}{2}}u(r_k \cdot)=U_1(\cdot)=c_{n,s}\left(\frac{1}{1+|\cdot|^{2-s}}\right)^{\frac{n-2}{2-s}}\hbox{ in }C^2_{\rm loc}(\rnp),$$
and 
$$r_{k+1}=o(\tau_{k+1})\hbox{ and }\tau_{k+1}=o(r_{k})\hbox{ as }k\to +\infty.$$
The $r_k$'s and $\tau_k$'s are the only critical points of $w$ in $(0,\delta_0]$ for some $\delta_0\in (0,1)$ small.
\end{prop}

\medskip\noindent{\it Proof of Proposition \ref{prop:construct:radii}:} We divide the proof into three steps. 
It follows from \eqref{hyp:control}, \eqref{hyp:inf:sup} and the Harnack inequality of Lemma \ref{lem:harnack} that there exists $C>0$ such that
\begin{equation}\label{bnd:w}
w(r)\leq C\hbox{ for all }r\in (0,1/2)\hbox{ and }\liminf_{r\to 0}w(r)=0.
\end{equation}
For $t>0$, we define
$$W(t):=w(e^{-t})=e^{-\frac{n-2}{2}t}\bu(e^{-t}).$$ 
Therefore, \eqref{bnd:w} rewrites
\begin{equation}\label{bnd:W}
W(t)\leq C\hbox{ for all }t\geq \ln 2\hbox{ and }\liminf_{t\to +\infty}W(t)=0.
\end{equation}

\medskip\noindent{\bf Step 1:} We claim that there exists $\eps_0>0$ such that 
\begin{equation}\label{w:convex}
\hbox{ for any }t\geq \ln 2, \hbox{ then }W(t)\leq \eps_0\,\Rightarrow W^{\prime\prime}(t)>0.
\end{equation}

\smallskip\noindent{\it Proof of the claim:} We use the conformal diffeomorphisme $\varphi$ defined in \eqref{def:phi}. Writing
$$u_\varphi(t,\theta):=e^{-\frac{n-2}{2}t}u(\varphi(t,\theta)),$$
for all $t>\ln 2$ and $\theta\in\mathbb{S}^{n-1}$, it follows from the invariance of the conformal Laplacian used in \eqref{conf:lap} that 
\begin{eqnarray*}
&&\left(\overline{-\Delta u}\right)\circ\varphi(t,\theta)= (-\Delta \bu)\circ \varphi(t,\theta)= L_{\varphi^\star \hbox{Eucl}}\bu\circ \varphi\\
&&=e^{\frac{n+2}{2}t}L_{dt^2+\hbox{can}_{n-1}} e^{-\frac{n-2}{2}t}\bu(e^{-t})=e^{\frac{n+2}{2}t}\left(-W^{\prime\prime}(t)+\frac{(n-2)^2}{4}W(t)\right).
\end{eqnarray*}
Independently, it follows from the Harnack inequality of Lemma \ref{lem:harnack} that
$$\left(\overline{\frac{u^{\crits-1}}{|x|^s}-\mu u^q}\right)(r)\leq C_1^{\crits-1}\frac{\bu(r)^{\crits-1}}{r^s}$$
for all $r\in (0,1/2)$. Therefore, equation \eqref{eq:50} yields
$$-W^{\prime\prime}(t)+\frac{(n-2)^2}{4}W(t)\leq C_1^{\crits-1}W^{\crits-1}$$
for $t>\ln 2$ large enough. We the get the conclusion by taking $\eps_0:=\left(\frac{(n-2)^2}{8 C_1^{\crits-1}}\right)^{\frac{1}{\crits-2}}$. This proves \eqref{w:convex}, and therefore the claim. This ends Step 1.\hfill$\Box$

\medskip\noindent{\bf Step 2:} We claim that there exists a sequence $(\tau_k)_k\in (0,1)$ that is decreasing, converging to $0$ and such that
$$\left\{r\in (0,1/2]/\, w'(r)=0\hbox{ and }w(r)\leq \eps_0\right\}=\{\tau_k/\, k\in\nn\}.$$
\smallskip\noindent{\it Proof of the claim:} Indeed, this set is at most countable since any critical point of this set  is nondegenerate due to \eqref{w:convex} (note that the $W^{\prime\prime}$ and $w^{\prime\prime}$ are proportional at a critical point). It is also infinite, since, otherwise, there exists $a>0$ such that $w$ has no critical point in $(0,a)$ below the threshold $\eps_0$, which then yields $\lim_{r\to 0}w(r)=0$ due to \eqref{bnd:w}. This then yields $|x|^{\frac{n-2}{2}}u(x)\to 0$ as $x\to 0$, contradicting \eqref{hyp:inf:sup}. This proves the claim and ends Step 2.\hfill$\Box$

\medskip\noindent{\bf Step 3:} We claim that for any $k\in\nn$, there exists a unique $r_k\in (\tau_{k+1},\tau_k)$ such that $w^\prime(r_k)=0$. More precisely, we have that
\begin{equation}\label{mono:w}
w'(r)>0\hbox{ for all }r\in (\tau_{k+1}, r_k)\hbox{ and }w'(r)<0 \hbox{ for all }r\in (r_k, \tau_k).
\end{equation}
Moreover,
\begin{equation}\label{lim:u:k}
\lim_{k\to +\infty}r_k^{\frac{n-2}{2}}u(r_k \cdot)=U_1(\cdot)=c_{n,s}\left(\frac{1}{1+|\cdot|^{2-s}}\right)^{\frac{n-2}{2-s}}\hbox{ in }C^2_{\rm loc}(\rnp).
\end{equation}
Note that no extraction of subsequence is required here.

\smallskip\noindent{\it Proof of the claim:} We define $r_k\in (\tau_{k+1},\tau_k)$ such that $$w(r_k)=\max \{w(r)/\, r\in [\tau_{k+1},\tau_k]\}.$$Since $w^\prime(\tau_{k+1})=w^\prime(\tau_k)=0$, it follows from \eqref{w:convex} that $w^{\prime\prime}(\tau_k),w^{\prime\prime}(\tau_{k+1})>0$. Therefore $\tau_{k+1}<r_k<\tau_k$, and $w^\prime(r_k)=0$.

\medskip\noindent We claim that $w(r_k)>\eps_0$. Otherwise, one has that $w(r)\leq \eps_0$ for all $r\in [\tau_{k+1},\tau_k]$, and therefore, it follows from \eqref{w:convex} that $W^{\prime\prime}(t)>0$ on $[-\ln\tau_k,-\ln\tau_{k+1}]$: this is a contradiction since $W^\prime$ vanishes at the boundary of this interval. 

\medskip\noindent We define $u_k(x):=r_k^{\frac{n-2}{2}}u(r_k x)$ for all $x\in B_{r_k^{-1}}(0)\setminus\{0\}$. It follows from Lemma \ref{lem:prelim:1} that, up to a subsequence, $(u_k)$ goes to $U\in C^2(\rnp)$ in $C^2_{\rm loc}(\rnp)$ as $k\to +\infty$. Since $w(r_k)\geq \eps_0$, we get that $\overline{u_k}(1)\geq \eps_0$, and then $\overline{U}(1)\geq \eps_0$, and then $U\not\equiv 0$. Therefore, it follows from Proposition \ref{prop:prelim:1} that there exists $\lambda>0$ such that $U\equiv U_\lambda$. The equality $w^\prime(r_k)=0$ rewrites as
$$\frac{d}{dr}\left(r^{\frac{n-2}{2}}\overline{u_k}(r)\right)_{r=1}=0.$$
Passing to the limit $k\to +\infty$ yields $\frac{d}{dr}\left(r^{\frac{n-2}{2}}\overline{U}(r)\right)_{r=1}=0$. Since $U\equiv U_\lambda$, the explicit computation of this derivative yields $\lambda=1$. Therefore, up to a subsequence, $u_k\to U_1$ in $C^2_{\rm loc}(\rnp)$ as $k\to +\infty$. Since the limit is unique, indeed, it holds for $k\to +\infty$ with no extraction. This proves \eqref{lim:u:k}.

\medskip\noindent We now prove that $r_k$ is the unique critical point of $w$ in $(\tau_{k+1},\tau_k)$. We define $w_k(r):=w(r_k r)$ for any $r>0$ and $k\in\nn$. It follows from the convergence of $(u_k)$ to $U_1$ that $$\lim_{k\to +\infty}w_k(r)=r^{\frac{n-2}{2}}U_1(r)=c_{n,s} \left(\frac{r^{\frac{2-s}{2}}}{1+r^{2-s}}\right)^{\frac{n-2}{2-s}}\hbox{ for all }r>0.$$
Moreover, this convergence holds in $C^1$. Therefore,
\begin{equation}\label{lim:wk:w}
\lim_{r\to 0}\lim_{k\to +\infty}w_k(r)=\lim_{r\to +\infty}\lim_{k\to +\infty}w_k(r)=0.
\end{equation}
Then, there exists $s_k, t_k>0$ such that for $k$ large enough
$$\left\{\begin{array}{l}
s_k<r_k<t_k\\
w(s_k)=w(t_k)=\eps_0\; ,\;w(r)>\eps_0\hbox{ for all }r\in (s_k,t_k)\\
w^\prime(r)>0\hbox{ for all }r\in [s_k, r_k)\\
w^\prime(r)<0\hbox{ for all }r\in (r_k, t_k]
\end{array}\right.$$
It then follows from the definition of $\tau_k$ and $\tau_{k+1}$ that $\tau_{k+1}<s_k$ and $t_k<\tau_k$. Moreover, since $w$ has no critical point below the level $\eps_0$ on the interval $(\tau_{k+1},\tau_k)$, we then get \eqref{mono:w}. This proves the uniqueness of a critical point in $(\tau_{k+1},\tau_k)$, and ends Step 3. \hfill$\Box$

\medskip\noindent As a remark, it follows from \eqref{lim:wk:w} that $\tau_{k+1}=o(r_k)$ as $k\to +\infty$, and therefore
\begin{equation}\label{lim:rk}
r_{k+1}=o(r_k) \hbox{ as }k\to +\infty.
\end{equation}

\section{Blow-up when $q<\crit-1$, Part II: Sharp pointwise estimate}\label{sec:blowup:2}
Here again, we let $u\in C^\infty(\ballp)$ be a positive solution to 
\begin{equation}\label{eq:60}
-\Delta u= \frac{u^{\crits-1}}{|x|^s}-\mu u^q \hbox{ in }\ballp
\end{equation}
such that there exists $C>0$ such that 
\begin{equation}\label{hyp:control:2}
|x|^{\frac{n-2}{2}}u(x)\leq C\hbox{ for all }x\in B_{1/2}(0)\setminus\{0\},
\end{equation}
and
\begin{equation}\label{hyp:inf:sup:2}
\liminf_{x\to 0}|x|^{\frac{n-2}{2}}u(x)=0\; ,\; \limsup_{x\to 0}|x|^{\frac{n-2}{2}}u(x)>0.
\end{equation}
We assume that $q<\crit-1$. The objective of this section is to prove the following sharp estimate:

\begin{prop}\label{prop:sharp:lim} Let  $u\in C^\infty(\ballp)$ be a positive solution to \eqref{eq:60} such that $q<\crit-1$, \eqref{hyp:control:2} and \eqref{hyp:inf:sup:2} hold. Then, for any $R>0$, for any $x\in B_{R r_k}(0)\setminus B_{R^{-1}r_{k+1}}(0)$, we have that
\begin{equation*}
u(x)=\left(1+\eps_k(x)\right)\left(c_{n,s}\left(\frac{r_{k+1}^{\frac{2-s}{2}}}{r_{k+1}^{2-s}+|x|^{2-s}}\right)^{\frac{n-2}{2-s}}+c_{n,s}\left(\frac{r_{k}^{\frac{2-s}{2}}}{r_{k}^{2-s}+|x|^{2-s}}\right)^{\frac{n-2}{2-s}}\right),
\end{equation*}
where $\lim_{k\to +\infty}\eps_k=0$ uniformly on $B_{R r_k}(0)\setminus B_{R^{-1}r_{k+1}}(0)$. Here, the $(r_k)$'s are as in Proposition \ref{prop:construct:radii}.
\end{prop}

\smallskip\noindent{\it Proof of Proposition \ref{prop:sharp:lim}:} We fix $R_0>0$. Let $(x_k)_k\in \ballp$ be such that $R_0^{-1}r_{k+1}\leq |x_k|\leq R_0 r_k$ for all $k\in\nn$. For convenience, we define $W_k(x):=U_{r_{k+1}}(x)+U_{r_k}(x)$ for all $x\in\rn$ and all $k\in\nn$, where $U_\lambda$ is defined in \eqref{U1-def} for all $\lambda>0$. Proposition \ref{prop:sharp:lim} is equivalent to prove that
\begin{equation}\label{lim:u:W:1}
\lim_{k\to +\infty}\frac{u(x_k)}{W_k(x_k)}=1.
\end{equation}
By uniqueness, it is enough to get the convergence for a subsequence. Therefore, in the sequel, we will systematically prove our results up to a subsequence. The proof of \eqref{lim:u:W:1} is divided into two steps. Our first step is to prove a control of $u$ that is almost optimal. This step will be used in Step 2.3.5 below. Note that when $\alpha=\beta=(n-2)/2$, then \eqref{est:epsilon} is \eqref{hyp:control:2}. The limiting case $(\alpha,\beta)=(n-2,0)$ will be proved in Step 2 and will yield Proposition \ref{prop:sharp:lim}. Step 2 is itself divided in three subcases.

\medskip\noindent{\bf Step 1:} We fix $\alpha, \beta\in (0,n-2)$. We fix $R_0>0$. We claim that there exists $C_{\alpha,\beta}(R_0)>0$ such that for any $k\in\nn$, we have that
\begin{equation}\label{est:epsilon}
u(x)\leq C_{\alpha,\beta}\left(\frac{r_{k+1}^{\alpha-\frac{n-2}{2}}}{|x|^\alpha}+\frac{r_{k}^{\beta-\frac{n-2}{2}}}{|x|^\beta}\right)\hbox{ for all }x\in B_{R_0 r_k}(0)\setminus B_{R_0^{-1}r_{k+1}}(0).
\end{equation}

\smallskip\noindent{\it Proof of the claim:} We define the elliptic operator $L\varphi:=-\Delta\varphi-\frac{u^{\crits-2}}{|x|^s}\varphi+\mu u^{q-1}\varphi$. Clearly $Lu=0$ on $\ballp$. We fix $\alpha\in(0,n-2)$. Using the Harnack inequality of Lemma \ref{lem:harnack}, we get that there exists $C_0>0$ such that
\neweq{eq:L:alpha}
 \begin{aligned}
L(|x|^{-\alpha})
& \geq |x|^{-\alpha-2} \left[\alpha(n-2-\alpha)-\left(|x|^{\frac{n-2}{2}}u(x)\right)^{\crits-2}\right]\\
&\geq  |x|^{-\alpha-2} \left[\alpha(n-2-\alpha)-C_0 w^{\crits-2}(|x|)\right].
\end{aligned} \endeq

\noindent It follows from \eqref{lim:wk:w} and the definition of $w_k$ that there exists $\rho_k\in (r_{k+1}, \tau_{k+1})$ and $\sigma_k\in (\tau_{k+1}, r_k)$ such that for $k\geq k_0$ large enough, 
\neweq{def:sigma:k}
\left\{ \begin{aligned}
& w^{\crits-2}(\rho_k)<\frac{\alpha(n-2-\alpha)}{C_0}\hbox{ and }\lim_{k\to +\infty}\frac{\rho_k}{r_{k+1}}=C_1>0\\
&w^{\crits-2}(\sigma_k)< \frac{\alpha(n-2-\alpha)}{C_0}\hbox{ and }\lim_{k\to +\infty}\frac{\sigma_k}{r_{k}}=C_2>0.
\end{aligned} \right. \endeq
In what follows, we let $k\geq k_0$. In particular, there exists $C>0$ such that 
$$ 
 u(x) \leq C \sigma_k^{-\frac{n-2}{2}} \  \text{for all } x\in\partial B_{\sigma_k}(0)\ \text{and } 
u(x)\leq C \rho_k^{-\frac{n-2}{2}} \text{for all } x\in\partial B_{\rho_k}(0).
$$
We fix $\beta\in (0,n-2)$. 
Up to taking $w(\rho_k)$ and $w(\sigma_k)$ smaller, we can assume that the inequalities in \eqref{def:sigma:k} also hold with $\beta$ instead of $\alpha$. 
Hence, \eqref{mono:w} yields that $$w^{\crits-2}<\min\left\{\frac{\alpha(n-2-\alpha)}{C_0}, \frac{\beta(n-2-\beta)}{C_0}\right\} \quad \text{on }  [\rho_k,\sigma_k].$$ 
Thus, \eqref{eq:L:alpha} gives that
$L(|x|^{-\alpha})>0$ and $L(|x|^{-\beta})>0$ for all $x\in B_{\sigma_k}(0)\setminus B_{\rho_k}(0)$.
Therefore, by setting $H(x):=C \rho_k^{\alpha-\frac{n-2}{2}}|x|^{-\alpha}+C \sigma_k^{\beta-\frac{n-2}{2}}|x|^{-\beta}$, we have that 
$$ \left\{ \begin{aligned}
& LH>0=Lu && \text{in } B_{\sigma_k}(0)\setminus B_{\rho_k}(0),&\\
& H\geq u &&\text{on } \partial(B_{\sigma_k}(0)\setminus B_{\rho_k}(0)).&
\end{aligned} \right.
$$ Using the comparison principle of Beresticky--Nirenberg--Varadhan \cite{bnv}, we find that 
\begin{equation}\label{ineq:u:H}
u(x)\leq H(x)\hbox{ for all }x\in B_{\sigma_k}(0)\setminus B_{\rho_k}(0).
\end{equation}
Up to taking $C$ larger, it follows from \eqref{def:sigma:k} and \eqref{hyp:control:2} that this inequality also holds on $B_{R_0 r_k}(0)\setminus B_{R_0^{-1}r_{k+1}}(0)$ for $k$ large. Clearly this also holds for any $k$. This proves \eqref{est:epsilon} and ends Step 1. \hfill$\Box$

\medskip\noindent{\bf Step 2:} We now prove \eqref{lim:u:W:1}. The proof is divided into three cases.

\medskip\noindent{\bf Case 2.1:} We assume that, up to a subsequence, $r_k=O(|x_k|)$ as $k\to +\infty$.

\smallskip\noindent {\it Proof of \eqref{lim:u:W:1} in Case 2.1.} Passing to a subsequence, we have $x_k=r_k\theta_k$ where $\lim_{k\to +\infty}\theta_k=\theta_\infty\neq 0$. Therefore, it follows from Proposition \ref{prop:construct:radii} that
\begin{equation*}
r_k^{\frac{n-2}{2}}u(x_k)=r_k^{\frac{n-2}{2}}u(r_k \theta_k)\to c_{n,s}\left(\frac{1}{1+|\theta_\infty|^{2-s}}\right)^{\frac{n-2}{2-s}}\hbox{ as }k\to +\infty.
\end{equation*} 
On the other hand, 
\begin{equation*}
r_k^{\frac{n-2}{2}}W_k(x_k)=O\left(\frac{r_{k+1}}{r_k}\right)^{\frac{n-2}{2}}+c_{n,s}\left(\frac{1}{1+|\theta_k|^{2-s}}\right)^{\frac{n-2}{2-s}}\ \ \text{as } k\to +\infty.
\end{equation*}
Hence, it follows from these two equalities and \eqref{lim:rk} that $u(x_k)=(1+o(1))\,W_k(x_k)$ as $k\to +\infty$. 
This proves \eqref{lim:u:W:1} in Case 2.1.\hfill$\Box$

\medskip\noindent{\bf Case 2.2:} Assume that, up to a subsequence, $x_k=O(r_{k+1})$ as $k\to \infty$. 

\smallskip\noindent {\it Proof of \eqref{lim:u:W:1} in Case 2.2.} One can proceed exactly as in Case 2.1. We omit the details.\qed

\medskip\noindent{\bf Case 2.3:} We assume that, up to a subsequence, $r_{k+1}=o(|x_k|)$ and $x_k=o(r_k)$ as $k\to +\infty$. 
Note that with this choice of $x_k$, we have that
\begin{equation}\label{lim:W:k}
W_k(x_k)=c_{n,s}(1+o(1))\left(r_{k+1}^{\frac{n-2}{2}}|x_k|^{2-n}+r_k^{-\frac{n-2}{2}}\right)\hbox{ as }k\to +\infty.
\end{equation}
We split the proof of \eqref{lim:u:W:1} in five steps.

\medskip\noindent{\bf Step 2.3.1:} We let $G$ be the Green's function of $-\Delta$ on $B_{1/2}(0)$ with Dirichlet boundary condition. We claim that
\begin{equation}\label{id:green}
u(x)=\int_{B_{1/2}(0)}G(x,y) f_{\mu,q}(y,u(y))\, dy-\int_{\partial B_{1/2}(0)}\partial_\nu G(x,y)u(y)\, d\sigma(y)
\end{equation}
for all $x\in B_{1/2}(0)\setminus \{0\}$. Here, $f_{\mu,q}$ was defined in \eqref{F-mu}. In particular, the right-hand side of this equation makes sense.

\medskip\noindent{\it Proof of the claim:} We fix $x\in B_{1/2}(0)\setminus \{0\}$, and we let $\delta>0$ be such that $\delta<|x|/2<1/4$. Green's Formula yields
\begin{eqnarray*}
u(x)&=&\int_{B_{1/2}(0)\setminus \overline{B_\delta(0)}}G(x,y)(-\Delta)u(y)\, dy\\
&&+\int_{\partial \left(B_{1/2}(0)\setminus \overline{B_\delta(0)}\right)}(-\partial_\nu G(x,y) u(y)+G(x,y)\partial_\nu u (y))\, d\sigma(y).
\end{eqnarray*}
Standard properties of the Green's function (see e.g. Robert \cite{r1}) yield the existence of $C>0$ such that 
\begin{equation}\label{ineq:G}
G(x,y)\leq C|x-y|^{2-n}\hbox{ and }|\nabla_y G(x,y)|\leq C|x-y|^{1-n}
\end{equation}
for all $x,y\in B_{1/2}(0)$, $x\neq y$. Using the pointwise control of Lemma \ref{lem:prelim:2} and \eqref{ineq:G}, we can pass to the limit as $\delta\to 0$ and get \eqref{id:green}. This proves the claim and ends Step 2.3.1.\hfill$\Box$

\medskip\noindent Since $|x_k|\to 0 $ and $W_k(x_k)\to +\infty$ as $k\to +\infty$, it follows that 
\neweq{plan} \frac{\int_{\partial B_{1/2}(0)}\partial_\nu G(x_k,y)\,u(y)\, d\sigma(y)}{W_k(x_k)}\to 0\ \ \text{as } k\to +\infty.
\endeq
In view of \eqref{plan} and \eqref{id:green}, to end the proof of \eqref{lim:u:W:1}, it remains to show that 
\neweq{cop}
\frac{\int_{B_{1/2}(0)}G(x_k,y)f_{\mu,q}(y,u(y))\, dy}{W_k(x_k)}\to 1\ \ \text{as } k\to +\infty.
\endeq
To this end, we notice that 
\neweq{sum} 
\int_{B_{1/2}(0)}G(x_k,y)f_{\mu,q}(y,u(y))\, dy=\sum_{j=1}^5 A_{j,k,R}(x_k),
\endeq
where for $j=1,2,3,4,5$, we define $A_{j,k,R}(x_k)$ as follows
\neweq{def:a}   
A_{j,k,R}(x_k):=\int_{D_{j,k,R}} G(x_k,y)f_{\mu,q}(y,u(y))\, dy.
\endeq
The domain of integration $D_{j,k,R}$ is given by 
$$ \left\{ \begin{aligned}
& D_{1,k,R}:=B_{1/2}(0)\setminus B_{Rr_k}(0), &&  D_{2,k,R}:=B_{R^{-1}r_{k+1}}(0),&\\
& D_{3,k,R}:=B_{R^{-1}r_k}(0)\setminus
B_{Rr_{k+1}}(0),&&  D_{4,k,R}:=B_{Rr_{k+1}}(0)\setminus B_{R^{-1}r_{k+1}}(0),&\\
&  D_{5,k,R}:=B_{Rr_k}(0)\setminus B_{R^{-1}r_k}(0).
\end{aligned} \right. $$
Shortly below, we shall prove the following claims:
\neweq{est}
\left\{\begin{aligned}
& \lim_{R\to +\infty} \lim_{k\to +\infty} \frac{A_{j,k,R}(x_k)}{W_k(x_k)}=0\quad \text{for } j=1,2,3;\\
& \lim_{R\to +\infty}  \lim_{k\to +\infty} \frac{A_{4,k,R}(x_k)}{r_{k+1}^{\frac{n-2}{2}} |x_k|^{2-n}}=
  \lim_{R\to +\infty}  \lim_{k\to +\infty} \frac{A_{5,k,R}(x_k)}{r_{k}^{-\frac{n-2}{2}}} =c_{n,s}. 
\end{aligned}\right.
\endeq
Then, the proof of \eqref{cop} follows from \eqref{lim:W:k}, \eqref{sum} and \eqref{est}. 

\medskip\noindent{\bf Step 2.3.2:} We claim that
 \begin{equation}\label{claim:1}
 \lim_{R\to +\infty} \lim_{k\to +\infty} \frac{A_{j,k,R}(x_k)}{W_k(x_k)}=0\quad \text{for } j=1,2.
\end{equation}

\smallskip\noindent{\it Proof of the claim: } 
We fix $R>0$. Recall that $Rr_k\leq |y|<1/2$ for every $y\in D_{1,k,R}$, whereas 
$|y|<R^{-1}r_{k+1}$ for any 
$y\in D_{2,k,R}$.  
Since $x_k=o(r_k)$ and $r_{k+1}=o(|x_k|)$ as $k\to +\infty$, using the pointwise bound in \eqref{ineq:G}, we find that 
\neweq{gk} G(x_k,y)\leq \left\{ \begin{aligned} 
&C |y|^{2-n} &&\text{ for all } y\in D_{1,k,R}&\\
& C |x_k|^{2-n} && \text{ for all } y\in D_{2,k,R}
\end{aligned} \right.\endeq
for $k$ large enough. Since $q\leq \crit-1$, \eqref{hyp:control:2} yields $|f_{\mu,q}(y,u(y))| \leq C |y|^{-\frac{n+2}{2}}$ for all $y\in B_{1/2}(0)\setminus \{0\}$. This inequality, \eqref{ineq:G} and the definition of $A_{j,k,R}$ in \eqref{def:a} yield
\neweq{duo}
\left\{ \begin{aligned}
& |A_{1,k,R}|
 \leq C \int_{B_{1/2}(0)\setminus B_{R r_k}(0)}|y|^{1-\frac{3n}{2}}\, dy\leq C (R r_k)^{-\frac{n-2}{2}}\\
 & |A_{2,k,R}| \leq C \int_{B_{R^{-1} r_{k+1}}(0)}|x_k|^{2-n}|y|^{-\frac{n+2}{2}}\, dy\leq C |x_k|^{2-n}(R^{-1} r_{k+1})^{\frac{n-2}{2}}
 \end{aligned} \right.
\endeq
for $k$ large. 
Using \eqref{duo} and \eqref{lim:W:k}, we arrive at \eqref{claim:1}. This ends Step 2.3.2.\hfill$\Box$

\medskip\noindent{\bf Step 2.3.3:} We claim that
\begin{equation}\label{claim:4}
\lim_{R\to +\infty}\lim_{k\to +\infty}\frac{A_{4,k,R}(x_k)}{r_{k+1}^{\frac{n-2}{2}}|x_k|^{2-n}}=c_{n,s}.
\end{equation}

\smallskip\noindent{\it Proof of the claim:} We denote $u_{k+1}(z):=r_{k+1}^{\frac{n-2}{2}}u(r_{k+1} z)$ and define $ \mathcal I_{R}$ as follows
\neweq{ir} \mathcal I_{R}:=  \int_{B_{R }(0)\setminus B_{R^{-1} }(0)}\frac{U_1^{\crits-1}(z)}{|z|^s}\, dz ,
\endeq
where $U_1$ is given by \eqref{def:U:lambda}. By Proposition \ref{prop:construct:radii}, we have
$u_{k+1}\to U_1$ in $C^2_{\rm loc} (\rn\setminus\{0\})$. Since $U_1$ satisfies 
$-\Delta U_1(z)=\frac{U_1^{\crits-1}(z)}{|z|^s}$ in $\rn\setminus\{0\}$, we obtain that
 \neweq{est:3} 
\begin{aligned} 
 \mathcal I_{R} &=-\int_{\partial B_{R }(0)}\partial_\nu U_1\, d\sigma+\int_{\partial B_{R^{-1} }(0)}\partial_\nu U_1\, d\sigma\\
&=(n-2) \,\omega_{n-1}\, c_{n,s} (1+R^{s-2})^{\frac{s-n}{2-s}} \left(1-R^{s-n}\right).
\end{aligned}\endeq
Using the change of variable $y=r_{k+1}z$, we find that $\frac{A_{4,k,R}(x_k)}{r_{k+1}^{\frac{n-2}{2}}|x_k|^{2-n}}$ is equal to 
$$
\int_{B_{R }(0)\setminus B_{R^{-1} }(0)}\frac{G(x_k,r_{k+1} z)}{|x_k|^{2-n}}\left(\frac{u_{k+1}^{\crits-1}(z)}{|z|^s}-\mu r_{k+1}^{\frac{n-2}{2}(\crit-1-q)}u_{k+1}^q(z)\right)\, dz. 
$$
It is standard (see Robert \cite{r1}) that
\begin{equation}\label{asymp:G}
\lim_{x,y\to 0}|x-y|^{n-2}G(x,y)=\frac{1}{(n-2)\omega_{n-1}},
\end{equation}
Since $r_{k+1}=o(|x_k|)$ as $k\to +\infty$, using \eqref{asymp:G} and $\mathcal I_R$ in \eqref{ir}, we get that
\neweq{est:2} 
\lim_{k\to \infty} \frac{A_{4,k,R}(x_k)}{r_{k+1}^{\frac{n-2}{2}}|x_k|^{2-n}}
=\frac{\mathcal I_R}{(n-2)\,\omega_{n-1}}.
\endeq
From  \eqref{est:3} and \eqref{est:2}, we conclude \eqref{claim:4}. 
This completes Step 2.3.3.\hfill$\Box$

\medskip\noindent{\bf Step 2.3.4:} We claim that
\begin{equation}\label{claim:5}
\lim_{R\to +\infty}\lim_{k\to +\infty}\frac{A_{5,k,R}(x_k)}
{r_{k}^{-\frac{n-2}{2}}}=c_{n,s}. 
\end{equation}

\noindent{\it Proof of the claim:} 
Since $x_k=o(r_k)$ as $k\to +\infty$, using \eqref{asymp:G}, we find that  
$$ \lim_{k\to +\infty} r_k^{n-2}G(x_k,r_{k} z)= \Gamma_n(z):=((n-2)\omega_{n-1})^{-1}|z|^{2-n}
$$ 
uniformly with respect to $z\in B_{R }(0)\setminus B_{R^{-1} }(0)$. Denoting $u_k(z):=r_k^{\frac{n-2}{2}}u(r_k z)$, then by 
Proposition \ref{prop:construct:radii},  
$u_{k}\to U_1$ in $C^2_{\rm loc} (\rn\setminus\{0\})$ as $k\to \infty$. 
By the change of variable $y=r_{k}z$, we find that $r_{k}^{\frac{n-2}{2}} A_{5,k,R}(x_k)$ equals
$$\int_{B_{R }(0)\setminus B_{R^{-1} }(0)}r_{k}^{n-2} G(x_k,r_{k} z)\left(\frac{u_k^{\crits-1}(z)}{|z|^s}-\mu r_k^{\frac{n-2}{2}(\crit-1-q)}u_k^q(z)\right)\, dz.
$$
Hence, letting $k\to \infty$, we get that
\neweq{eq:23}
\lim_{k\to +\infty}\frac{A_{5,k,R}(x_k)}
{r_{k}^{-\frac{n-2}{2}}}
=\int_{B_{R }(0)\setminus B_{R^{-1} }(0)}\Gamma_n(z)\,\frac{U_1^{\crits-1}(z)}{|z|^s}\, dz:=\mathcal J_R.
\endeq
Using Green's representation formula, equation \eqref{eq:lim:2} satisfies by $U_1$ and the explicit expression of $U_1$, we see that as $R\to +\infty$
\neweq{eq:24}
\mathcal J_R \to\int_{\rn}\Gamma_n(z)\,\frac{U_1^{\crits-1}(z)}{|z|^s}\, dz
 =-\int_{\rn}\Gamma_n(z) \,\Delta U_1(z)\, dz
=U_1(0)=c_{n,s}.
\endeq
Note that this computation makes sense due to the growth of $U_1$. From \eqref{eq:23} and \eqref{eq:24}, we obtain \eqref{claim:5}. This proves the claim and ends Step 2.3.4.\hfill$\Box$

\medskip\noindent{\bf Step 2.3.5:} We claim that
\begin{equation}\label{claim:3}
\lim_{R\to +\infty}\lim_{k\to +\infty}\frac{A_{3,k,R}(x_k)}{W_k(x_k)}=0.
\end{equation}
\medskip\noindent{\it Proof of the claim:} 
We first show that 
\neweq{ch} 
\lim_{|x|\to 0} |x|^s [u(x)]^{q-(\crits-1)}=0. 
\endeq  
When $\crits-1< q<\crit -1$, then \eqref{ch} follows from \eqref{hyp:control:2}. If $q=\crits-1$, then \eqref{ch} is clear. 
If $1< q<\crits-1$, then 
for every $m>0$
\neweq{mbb} \Delta (|x|^m)\geq |x|^{mq}\quad \text{for all } 0<|x|< \left[m(m+n-2)\right]^{\frac{1}{(q-1)m+2}}. \endeq
From \eqref{hyp:inf:sup:2}, we have $\limsup_{x\to 0} u(x)=+\infty$. The spherical Harnack inequality \eqref{harnack} gives 
a sequence of positive numbers $\{\xi_k\}_{k\geq 1}$ decreasing to $0$ such that 
$$ u(x)\geq 1 \quad \text{for all } |x|=\xi_k\ \ \text{and every } k\geq 1.$$ 
Without loss of generality, we assume that $\xi_1<1/2$. 
Let $k_1>1$ be large such that $\xi_{k_1}< \left[m(m+n-2)\right]^{\frac{1}{(q-1)m+2}}$. Since $\Delta u\leq u^q$ in $B_1(0)\setminus\{0\}$, using 
\eqref{mbb} and the comparison principle (see for instance Lemma 2.1 in C\^{\i}rstea-R\u adulescu \cite{cr}), we find that  
$u(x)\geq |x|^m$ for all $\xi_k\leq |x|\leq \xi_{k_1}$ and $k>k_1$. Letting $k\to \infty$ and choosing $ 0<m<s/(\crits-1-q)$, we conclude 
\eqref{ch}.  

\noindent Hence, there exists a constant $C_0> 0$ such that
\begin{equation}\label{bnd:u:q}
u^q(x)\leq C_0 \frac{u^{\crits-1}(x)}{|x|^s}\hbox{ for all }x\in B_{1/2}(0)\setminus\{0\}.
\end{equation}
Since $s\in (0,2)$, we can choose $\alpha$ and $\beta$ in the interval $(0,n-2)$ such that
\begin{equation}\label{choice:a:b}
\beta \gamma+s <2\  \text{and} \ n<\alpha \gamma+s,\quad \text{where } \gamma=\crits-1.
\end{equation}
We fix $R_0>0$. We let $R>R_0^{-1}$. Since $D_{3,k,R}:=B_{R^{-1} r_k}(0)\setminus B_{R r_{k+1}}(0)$, we have $D_{3,k,R}\subset B_{R_0r_k}(0)\setminus B_{R_0^{-1} r_{k+1}}(0)$. 
Using \eqref{bnd:u:q} and the definition of $A_{3,k,R}(x_k)$ in \eqref{def:a}, we find that
$$ |A_{3,k,R}(x_k) | \leq C' \int_{D_{3,k,R}} G(x_k,y) \frac{u^{\crits-1}(y)}{|y|^s} \,dy 
$$ for some constant $C'>0$. 
We define $M_{k,R}(x_k)$ and $N_{k,R}(x_k)$ as follows
\neweq{mnm}
\left\{ \begin{aligned}
& M_{k,R}(x_k):= r_{k+1}^{\left(\alpha-\frac{n-2}{2}\right)\gamma} \,  \int_{D_{3,k, R}} \frac{|x_k-y|^{2-n}}{|y|^{\alpha\gamma+s}}\, dy,\\
&  N_{k,R}(x_k):= r_{k}^{\left(\beta-\frac{n-2}{2}\right)\gamma} \,
 \int_{D_{3,k, R}} \frac{|x_k-y|^{2-n}}{|y|^{\beta\gamma+s}}\, dy.
\end{aligned} \right.
\endeq

\smallskip\noindent  By \eqref{ineq:G} and \eqref{est:epsilon}, there exists $C>0$ (independent of $R>R_0^{-1}$) such that 
\neweq{est:1}
 |A_{3,k,R}(x_k) | \leq C(M_{k,R}(x_k)+N_{k,R}(x_k))\quad \text{for all } k\geq 1. 
 \endeq

\smallskip\noindent We claim that there exist positive constants $C$, $\tau$ and $\tau'$ such that as $k\to \infty$
\neweq{claim:int:1}
\left\{ \begin{aligned}
&  M_{k,R}(x_k)
\leq Cr_{k+1}^{\frac{n-2}{2}}|x_k|^{2-n} \left(R^{-\tau}+o(1)\right),\\
& N_{k,R}(x_k)\leq C R^{-\tau'} r_k^{-\frac{n-2}{2}} \left(1+o(1)\right) .
\end{aligned} \right. 
\endeq
We prove the claim. In what follows, we take $k\geq k_0$ and denote
$$ \left\{ \begin{aligned}
& T_{1,k,R}(x_k):= \int_{D_{3,k,R}\cap\{|x_k-y|\geq |x_k|/2\}}\frac{|x_k-y|^{2-n}}{|y|^{\alpha \gamma+s}}\, dy,\\
& T_{2,k,R}(x_k):= \int_{D_{3,k,R}\cap\{|x_k-y|< |x_k|/2\}}\frac{|x_k-y|^{2-n}}{|y|^{\alpha \gamma+s}}\, dy.
\end{aligned} \right.
$$
Let $\tau:=\alpha \gamma+s-n$. Using that $n<\alpha \gamma+s$, we have $\tau>0$ and 
\neweq{t1}  T_{1,k,R}(x_k) \leq \left(\frac{|x_k|}{2}\right)^{2-n}\int_{B_1(0)\setminus B_{R r_{k+1}}(0)}\frac{dy}{|y|^{\alpha \gamma+s}}\leq C|x_k|^{2-n}\left(R r_{k+1}\right)^{-\tau}
\endeq
for some constant $C>0$. On the other hand, we find that 
\neweq{t2} T_{2,k,R}(x_k)\leq C|x_k|^{-(\alpha \gamma+s)} \int_{\{|x_k-y|< |x_k|/2\}}|x_k-y|^{2-n}\, dy\leq C|x_k|^{2-(\alpha \gamma+s)}. 
\endeq
Using \eqref{mnm}, \eqref{t1} and \eqref{t2}, we conclude the first inequality in \eqref{claim:int:1} since
$$
\begin{aligned}
M_{k,R}(x_k)
&\leq  r_{k+1}^{\left(\alpha-\frac{n-2}{2}\right)\gamma}\left(T_{1,k,R}(x_k)+ T_{2,k,R}(x_k)\right)\\
&\leq C r_{k+1}^{\frac{n-2}{2}}\,|x_k|^{2-n} \left[ R^{-\tau} + \left(\frac{r_{k+1}}{|x_k|}\right)^{\tau}
\right].
\end{aligned}
$$

\noindent For the second estimate in \eqref{claim:int:1}, we denote $\tau':=2-\beta\gamma-s$. From the choice of $\beta$ in  \eqref{choice:a:b}, we have $\tau'>0$. 
With the change of variable $y=R^{-1}r_k z$, we find that  
\neweq{nkr}
N_{k,R}(x_k) \leq r_{k}^{2-s-\frac{n-2}{2}\gamma}R^{-\tau'}\int_{B_1(0)}\left|R\frac{x_k}{r_k}-z\right|^{2-n}|z|^{\tau'-2}\, dz.
\endeq
The integral in the right-hand side of \eqref{nkr} converges as $k\to +\infty$ since $\tau'>0$ and $x_k=o(r_k)$ as $k\to +\infty$ 
From $\gamma=\crits-1$, we have $2-s-\frac{n-2}{2}\gamma=\frac{2-n}{2}$. Thus \eqref{nkr} shows  
the second inequality in \eqref{claim:int:1}. This proves the claim of \eqref{claim:int:1}.

\vspace{0.2cm}\noindent Using \eqref{claim:int:1} into \eqref{est:1}, jointly with \eqref{lim:W:k}, we get \eqref{claim:3}. This ends Step 2.3.5.\hfill$\Box$

\smallskip\noindent {\it Proof of \eqref{lim:u:W:1} in Case 2.3.} This is a consequence of Steps 2.3.1 to 2.3.5 above. This ends the proof of Proposition  \ref{prop:sharp:lim}.\qed

\medskip\noindent As a consequence of Proposition \ref{prop:sharp:lim}, we get the following:

\begin{prop}\label{prop:sharp:lim:opt} Let $q<\crit-1$ and  $u\in C^\infty(\ballp)$ be a positive solution to \eqref{eq:60} such that \eqref{hyp:control:2} and \eqref{hyp:inf:sup:2} hold. Then we have that
\begin{equation}\label{sharp:asymp:opt}
u(x)=(1+o(1))\sum_{k=0}^\infty U_{r_k}(x)\quad \text{as } x\to 0.
\end{equation}
In particular, $u$ develops a  singularity of (MB) type.
\end{prop}

\begin{proof}[Proof of Proposition \ref{prop:sharp:lim:opt}] We start with a preliminary remark. Since $r_{k+1}=o(r_k)$ as $k\to +\infty$, we have $r_k-r_{k+1}=(1+o(1))r_k>0$ as $k\to +\infty$. Therefore, since $r_k\to 0$ as $k\to +\infty$, we find that
\begin{equation}\label{asymp:r:1}
r_{l+2}^{\frac{n-2}{2}}=\sum_{k=l+2}^\infty \left(r_k^{\frac{n-2}{2}}-r_{k+1}^{\frac{n-2}{2}}\right)=(1+o(1))\sum_{k=l+2}^\infty r_k^{\frac{n-2}{2}}
\end{equation}
as $l\to +\infty$. Similarly, we obtain that
\begin{equation}\label{asymp:r:2}
\left(1+o(1)\right)r_{l-1}^{-\frac{n-2}{2}}=\sum_{k=1}^{l-1} \left(r_{k}^{-\frac{n-2}{2}}-r_{k-1}^{-\frac{n-2}{2}}\right)=(1+o(1))\sum_{k=1}^{l-1} r_{k}^{-\frac{n-2}{2}}
\end{equation}
as $l\to +\infty$. Let $x\in \ballp$ be such that $|x|<r_{0}$. Let $l\in\nn$ be such that
\neweq{inx} r_{l+1}\leq |x|< r_{l}.\endeq
Using \eqref{inx} and the definition of $U_\lambda$ in \eqref{dicho:ui}, we see that 
\neweq{alb} U_{l+1}(x)\geq 2^{\frac{n-2}{s-2}}  |x|^{2-n}r_{l+1}^{\frac{n-2}{2}} \ \text{ and }\ 
U_l(x)\geq 2^{\frac{n-2}{s-2}}  r_{l}^{-\frac{n-2}{2}} .
\endeq
Using \eqref{asymp:r:1} and \eqref{asymp:r:2}, we get that 
\neweq{ali} 
 \left\{ \begin{aligned}
& \sum_{k=l+2}^\infty U_{r_k}(x)
\leq c_{n,s} |x|^{2-n}\sum_{k=l+2}^\infty r_k^{\frac{n-2}{2}}\leq 2 c_{n,s}\,|x|^{2-n}r_{l+2}^{\frac{n-2}{2}},\\
& 
\sum_{k=0}^{l-1} U_{r_k}(x)
\leq c_{n,s} \sum_{k=0}^{l-1}  r_k^{-\frac{n-2}{2}}\leq 2 c_{n,s} r_{l-1}^{-\frac{n-2}{2}}
\end{aligned} \right. 
\endeq
for $l$ large enough. From \eqref{alb} and \eqref{ali}, we find $C>0$ such that 
$$
\sum_{k=0}^{l-1} U_{r_k}(x)+  \sum_{k=l+2}^\infty U_{r_k}(x)
\leq C \left(\frac{r_{l+2}}{r_{l+1}}\right)^{\frac{n-2}{2}} U_{l+1}(x)+ C  \left(\frac{r_{l}}{r_{l-1}}\right)^{\frac{n-2}{2}} U_l(x)
$$
for $l$ large enough. Since for $|x|$ small enough, $l$ is large, we obtain that
\neweq{ye} \sum_{k=0}^{l-1} U_{r_k}(x)+  \sum_{k=l+2}^\infty U_{r_k}(x)\leq \eps_l \left(U_{l+1}(x)+ U_{l}(x)\right),
\endeq
where
$$\eps_l:=C \left(\left(\frac{r_{l+2}}{r_{l+1}}\right)^{\frac{n-2}{2}}+ \left(\frac{r_{l}}{r_{l-1}}\right)^{\frac{n-2}{2}} \right)\to 0\ \text{as } l\to +\infty.$$
By Proposition~\ref{prop:sharp:lim:opt}, we have
\neweq{by} u(x)=(1+\varepsilon_l(x)) \left(U_{r_{l+1}}(x)+U_{r_l}(x)\right),\endeq
where $\lim_{l\to +\infty} \varepsilon_l(x)=0$ uniformly with respect to $x$ in $B_{r_l}(0)\setminus B_{r_{l+1}}(0)$. 
From \eqref{ye} and  \eqref{by}, we conclude \eqref{sharp:asymp:opt} and therefore Proposition \ref{prop:sharp:lim:opt}.\end{proof}

\section{Estimate for the radii $(r_k)$}\label{sec:radii}
The objective of this section is to prove the following asymptotics.

\begin{prop}\label{prop:est:rk} Let $u\in C^\infty(\ballp)$ be a positive solution to 
\begin{equation}\label{eq:70}
-\Delta u= \frac{u^{\crits-1}}{|x|^s}-\mu u^q \hbox{ in }\ballp.
\end{equation}
We assume that $q<\crit-1$ and 
\begin{equation}\label{inf-sup}
\liminf_{x\to 0}|x|^{\frac{n-2}{2}}u(x)=0\quad \text{and}\quad \limsup_{x\to 0}|x|^{\frac{n-2}{2}}u(x)\in (0,\infty).
\end{equation}
Then
\neweq{q-ineq} q>\crit-2=\frac{4}{n-2}.\endeq
We let $(r_k)_k$ the points of local maxima and $(\tau_k)_k$ the points of local minima of $w$ defined in \eqref{def:w} and Proposition \ref{prop:construct:radii}. Then
as $k\to +\infty$, we have
\begin{equation}\label{est:r}
\tau_{k+1}=\left(1+o(1)\right)\sqrt{r_{k+1}r_k}\ \text{and}\ 
r_{k+1}=\left(K+o(1)\right)r_{k}^{\frac{1}{q-(\crit-2)}}, 
\end{equation}
where $K$ is a positive constant defined by 
\neweq{def-k} K=\left(\frac{(\crit-1-q)\mu }{(q+1)(n-2) c_{n,s}^2\omega_{n-1}}\int_{\rn} U_1^{q+1}\, dx\right)^{\frac{2}{(n-2)(q-(\crit-2))}}.\endeq
\end{prop}

\begin{proof}[Proof of Proposition \ref{prop:est:rk}] 
We define $\lambda_k:=\sqrt{r_{k+1}r_k}$. 
By Proposition \ref{prop:prelim:1}, we have $P^{(q)}(u)=0$ so that by letting $r_1\to 0$ and $r_2=\lambda_k$ in 
\eqref{id:poho:invar}, we find that 
\begin{equation}\label{eq:P:lambda}
P_{\lambda_k}^{(q)}(u)=\frac{(n-2)}{2(q+1)}(\crit-1-q)\mu \int_{B_{\lambda_k}(0)} u^{q+1}\, dx.
\end{equation}

\noindent We divide the proof of Proposition \ref{prop:est:rk} into four steps. The first assertion of \eqref{est:r} is proved in Step~1. 
The left-hand side of \eqref{eq:P:lambda} is estimated in \eqref{est:P:1}, see Step~2.
Then, in Step~3, we prove $q>2/(n-2)$, which gives that $U_1\in L^{q+1}(\rn)$. We estimate $ \int_{B_{\lambda_k}(0)} u^{q+1}\, dx$ in \eqref{vam}, see Step~4.  
From \eqref{eq:P:lambda}, \eqref{est:P:1} and \eqref{vam}, we conclude the second claim of \eqref{est:r}, which implies \eqref{q-ineq} since $r_k\to 0$ as
$k\to +\infty$.        

\smallskip\noindent{\bf Step 1:} We claim that
\begin{equation}\label{lim:tau}
\lim_{k\to +\infty}\frac{\tau_{k+1}}{\lambda_k}=1.
\end{equation}
\smallskip\noindent{\it Proof of the claim:} 
Since $r_{k+1}=o(r_k)$ as $k\to +\infty$, we see that $r_{k+1}=o(\lambda_k)$ and $\lambda_k=o(r_k)$ as $k\to +\infty$. 
For any $k\in \nn$, we define
\begin{equation}\label{def:tuk}
\tilde{u}_k(x):=r_k^{\frac{n-2}{2}}u(\lambda_k x)\hbox{ for }x\in B_{1/\lambda_k}(0)\setminus\{0\}.
\end{equation}
We show that 
\neweq{lim:u:t}
\lim_{k\to +\infty}\tilde{u}_k(x)=\tilde{u}(x):=c_{n,s}\left(|x|^{2-n}+1\right)\quad \text{in } C^2_{\rm loc}(\rn\setminus\{0\}).
\endeq
Using the pointwise control of Proposition \ref{prop:sharp:lim}, we obtain that $u_k(x)\to \tilde u(x)$ for all $x\in\rn\setminus\{0\}$.  
Moreover, equation \eqref{eq:70} rewrites
$$-\Delta \tilde{u}_k=\left(\frac{\lambda_k}{r_k}\right)^{2-s}\frac{\tilde{u}_k^{\crits-1}}{|x|^s}-\mu r_k^{\frac{n-2}{2}(\crit-1-q)}\left(\frac{\lambda_k}{r_k}\right)^{2}\tilde{u}_k^q
\quad \text{in } B_{1/\lambda_k}(0)\setminus\{0\}. $$
Using this equation, \eqref{lim:u:t} and the elliptic theory, we obtain \eqref{lim:u:t}. 

\vspace{0.2cm}
\noindent Let $w$ be given by \eqref{def:w}. We define
\neweq{clock} \tilde{w}_k(r):=\left(\frac{r_k}{\lambda_k}\right)^{\frac{n-2}{2}}w(\lambda_k r)=r^{\frac{n-2}{2}}\overline{\tilde{u}_k}(r)\quad \text{for all } r>0.\endeq
Passing to the limit in \eqref{clock} and using \eqref{lim:u:t}, we get that
\neweq{ncon} \lim_{k\to +\infty}\tilde{w}_k(r)=c_{n,s}(r^{-\frac{n-2}{2}}+r^{\frac{n-2}{2}})\quad \hbox{for all }r>0.\endeq
Moreover, the convergence in \eqref{ncon} holds in $C^2_{\rm loc}(\rr\setminus\{0\})$. Since $r\longmapsto r^{-\frac{n-2}{2}}+r^{\frac{n-2}{2}}$ has a nondegenerate local minimum point at $r=1$, then for $k$ large, $\tilde{w}_k$ admits a critical point $\rho_k$ such that $\lim_{k\to +\infty}\rho_k=1$. Thus, for $k$ large, $w$ admits a nondegenerate local minimum at $\lambda_k \rho_k$. 
We have $r_{k+1}<\lambda_k \rho_k<r_k$ for $k$ large enough. Hence, from the uniqueness of the critical points in 
Proposition~\ref{prop:construct:radii}, we find that $\tau_{k+1}=\lambda_k \rho_k$. This yields \eqref{lim:tau} and ends Step 1.\qed 
 
\medskip\noindent{\bf Step 2:} We claim that  
\begin{equation}\label{est:P:1}
P_{\lambda_k}^{(q)}(u)=\left(\frac{(n-2)^2}{2}c_{n,s}^2\omega_{n-1}+o(1)\right)\left(\frac{\lambda_k}{r_k}\right)^{n-2}\quad \text{as } k\to +\infty. 
\end{equation}
\smallskip\noindent{\it Proof of the claim:} 
For $\tilde u$ given by \eqref{lim:u:t}, 
a straightforward computation yields that
\neweq{mcv} \int_{\partial B_{1}(0)}\left[(x,\nu)\frac{|\nabla \tilde{u}|^2}{2}-T(x,\tilde u)\,\partial_\nu \tilde{u}\right]\, d\sigma=\frac{(n-2)^2}{2}c_{n,s}^2\,\omega_{n-1}.\endeq
From the limit \eqref{lim:u:t} and $\lambda_k=o(r_k)$ as $k\to +\infty$, we get that
$$  \mathcal G_k(x,\tilde u):=\left(\frac{\lambda_k}{r_k}\right)^{2-s}\frac{\tilde{u}_k^{\crits}}{\crits|x|^s}- \mu \left(\frac{\lambda_k}{r_k}\right)^{2}r_k^{\frac{n-2}{2}(\crit-1-q)}\frac{\tilde{u}_k^{q+1}}{q+1}\to 0\ \text{as } k\to \infty
$$ uniformly with respect to $x\in \partial B_1(0)$. 
Hence, 
using the definition \eqref{def:P:invar} of the Pohozaev-type integral, the definition \eqref{def:tuk} of $\tilde{u}_k$ and a change of variable, we have
\begin{eqnarray*}
P_{\lambda_k}^{(q)}(u)&=&\int_{\partial B_{\lambda_k}(0)}\left[(x,\nu)\left(\frac{|\nabla u|^2}{2}-\frac{u^{\crits}}{\crits|x|^s}+\mu\frac{u^{q+1}}{q+1}\right)
- T(x,u)\,\partial_\nu  u\right]\, d\sigma\\
&=& \left(\frac{\lambda_k}{r_k}\right)^{n-2}\int_{\partial B_{1}(0)}\left[(x,\nu)\left(\frac{|\nabla \tilde{u}_k|^2}{2}-
 \mathcal G_k(x,\tilde u)\right) -T(x,\tilde u_k)\,\partial_\nu \tilde{u}_k
\right]\, d\sigma\\
&=& \left(\frac{\lambda_k}{r_k}\right)^{n-2}\left(\int_{\partial B_{1}(0)}\left[(x,\nu)\frac{|\nabla \tilde{u}|^2}{2}-T(x,\tilde u)\,\partial_\nu \tilde{u}
\right]\, d\sigma+o(1)\right)
\end{eqnarray*}
as $k\to +\infty$. This, jointly with \eqref{mcv} proves \eqref{est:P:1}. This ends Step~2. \qed

\medskip\noindent{\bf Step 3:} We claim that $q>2/(n-2)$ and $U_1\in L^{q+1}(\rn)$. 

\smallskip\noindent{\it Proof of the claim:}  From $ \limsup_{x\to 0}|x|^{\frac{n-2}{2}}u(x)<\infty$,   
there exists $C>0$ such that $|x|^{\frac{n-2}{2}}u(x)\leq C$ for all $x\in B_{1/2}(0)\setminus\{0\}$. 
Since $q<\crit-1$, we find that 
\begin{equation}\label{est:P:2}
\mathcal T_{1,k,R} \leq C  (R^{-1}r_{k+1})^{\frac{n-2}{2}(\crit-1-q)},\quad \text{where }  \mathcal T_{1,k,R}:=\int_{B_{R^{-1}r_{k+1}}(0)} u^{q+1}\, dx 
\end{equation}
for any $R>0$. We denote
\neweq{nota} 
\left\{ \begin{aligned}
& \mathcal T_{2,k,R}:=\int_{B_{Rr_{k+1}}(0)\setminus B_{R^{-1}r_{k+1}}(0)} u^{q+1}\, dx,\\ 
& \mathcal T_{3,k,R}:=\int_{B_{\lambda_k}(0)\setminus B_{Rr_{k+1}}(0)} u^{q+1}\, dx. 
\end{aligned} \right. \endeq
By Proposition \ref{prop:construct:radii}, $u_{k}\to U_1$ in $C^2_{\rm loc}(\rn\setminus\{0\})$, where $u_{k}(z):=r_{k}^{\frac{n-2}{2}}u(r_{k} z)$. Hence, 
\neweq{est:P:3}
\begin{aligned}
\mathcal T_{2,k,R} &=r_{k+1}^{\frac{n-2}{2}(\crit-1-q)}\int_{B_R(0)\setminus B_{R^{-1}}(0)}u_{k+1}^{q+1}\, dz\\
&=r_{k+1}^{\frac{n-2}{2}(\crit-1-q)}\left(\int_{B_R(0)\setminus B_{R^{-1}}(0)}U_1^{q+1}\, dz+o(1)\right)
\end{aligned}\endeq
as $k\to +\infty$. Using the optimal control of Proposition \ref{prop:sharp:lim}, we find that
\neweq{est:P:4}
\begin{aligned}
\mathcal T_{3,k,R}  &\leq  C\int_{B_{\lambda_k}(0)\setminus B_{Rr_{k+1}}(0)} \left(r_{k+1}^{\frac{n-2}{2}(q+1)}|x|^{-(n-2)(q+1)}+r_k^{-\frac{n-2}{2}(q+1)}\right)\, dx\\
&\leq  C r_{k+1}^{\frac{n-2}{2}(q+1)}\int_{Rr_{k+1}}^{\lambda_k}r^{1-(n-2)q}\, dr+ C \lambda_k^n r_k^{-\frac{n-2}{2}(q+1)}\ \text{as } k\to +\infty. 
\end{aligned} \endeq

\noindent Assume by contradiction that $q\leq 2/(n-2)$. By \eqref{est:P:2}, \eqref{est:P:3} and \eqref{est:P:4}, we have
$$
\int_{B_{\lambda_k}(0)} u^{q+1}\, dx\leq Cr_{k+1}^{\frac{n-2}{2}(\crit-1-q)}+ C r_{k+1}^{\frac{n-2}{2}(q+1)}
\times \left\{ 
\begin{aligned} 
& \ln \frac{\lambda_k}{r_{k+1}} && \text{if } q=\frac{2}{n-2}&\\
& \lambda_k^{n-(n-2)(q+1)} && \text{if } q<\frac{2}{n-2}&
\end{aligned} \right.
$$
as $k\to \infty$. Combining this inequality with \eqref{eq:P:lambda} and \eqref{est:P:1}, we get that
$$ \left(\frac{\lambda_k}{r_k}\right)^{n-2}\leq \left\{ \begin{aligned} 
& C r_{k+1}^{\frac{n}{2}} \ln \frac{\lambda_k}{r_{k+1}} && \text{if } q=\frac{2}{n-2}&\\
& C r_{k+1}^{\frac{n-2}{2}(\crit-1-q)}+ C r_{k+1}^{\frac{n-2}{2}(q+1)}\lambda_k^{n-(n-2)(q+1)} && \text{if } q<\frac{2}{n-2}&
\end{aligned} \right. $$
as $k\to \infty$
Then, since $\lambda_k=\sqrt{r_k r_{k+1}}$, we obtain that
$$1\leq C r_k^{\frac{n-2}{2}}r_{k+1}\times 
\left\{ \begin{aligned} 
& \left(\ln r_k-\ln r_{k+1}\right) 
&& \text{if } q=\frac{2}{n-2}&\\ 
& \left(r_k^{\frac{2-(n-2)q}{2}}+r_{k+1}^\frac{2-(n-2)q}{2}\right) 
&& \text{if } q<\frac{2}{n-2}& 
\end{aligned} \right. 
$$
as $k\to +\infty$, which is a contradiction since $r_k\to 0$ as $k\to +\infty$. 

\smallskip\noindent Hence, $q>2/(n-2)$, which yields that $U_1\in L^{q+1}(\rn)$, concluding Step 3.\qed

\medskip\noindent{\bf Step 4:} We claim that 
\neweq{vam}
\int_{B_{\lambda_k}(0)} u^{q+1}\, dx=r_{k+1}^{\frac{n-2}{2}(\crit-1-q)}\left(\int_{\rn}U_1^{q+1}\, dx+o(1)\right)\ \text{as } k\to +\infty. 
\endeq

\smallskip\noindent{\it Proof of the claim:} 
Since $q>2/(n-2)$, inequality \eqref{est:P:4} yields
\neweq{est:P:5}
\mathcal T_{3,k,R}
\leq  r_{k+1}^{\frac{n-2}{2}(\crit-1-q)}\left(C R^{2-(n-2)q}+ C \left(\frac{r_{k+1}}{r_k}\right)^{\frac{(n-2)q-2}{2}}\right).
\endeq
Recall that $\mathcal T_{i,k,R}$ with $i=1,2,3$ are given by
\eqref{est:P:2} and \eqref{nota}. We have $$\int_{B_{\lambda_k}(0)} u^{q+1}\, dx=\mathcal T_{1,k,R}+\mathcal T_{2,k,R}+\mathcal T_{3,k,R} \quad \text{for all } R>0.$$ 
Letting $k\to +\infty$ and then $R\to +\infty$ in \eqref{est:P:2}, \eqref{est:P:3} and \eqref{est:P:5}, we get \eqref{vam}. \qed

\smallskip
\noindent This completes the proof of Proposition \ref{prop:est:rk}. \end{proof}

\section{Proof of Theorem \ref{thm:1}}

Let $u\in C^\infty(B_1(0)\setminus \{0\})$ be a positive solution to \eqref{eq:u}. 
If $q>\crit-1$, then by Corollary~\ref{coro:remove}, zero is a removable singularity. 
If $\crits-1<q<\crit-1$, then by Proposition \ref{prop:lim:1}, the solution $u$ either develops a (ND) profile, or 
\begin{equation}\label{pf:1}
\limsup_{x\to 0} |x|^{\frac{n-2}{2}}u(x)<+\infty. 
\end{equation}
If $q\leq\crits-1$, then Proposition \ref{prop:bnd:u} gives that \eqref{pf:1} also holds.

\medskip\noindent Now, for $q<\crit-1$, assuming \eqref{pf:1}, it follows from Propositions \ref{prop:remove}, \ref{prop:sharp:lim:opt} and \ref{prop:est:rk} that either zero is a removable singularity, or $u$ develops a singularity of type (MB), or there exist positive constants $c_1$ and $c_2$ such that
\begin{equation}\label{pf:2}
c_1\leq |x|^{\frac{n-2}{2}}u(x)\leq c_2\hbox{ for all }x\in B_{1/2}(0)\setminus \{0\}.
\end{equation}
We assume that \eqref{pf:2} holds. Since $q<\crit-1$, following step by step the proof of Theorem 4.1 in Hsia--Lin--Wang \cite{hlw} (pages 1642 to 1648), one gets that $u$ develops a singularity of (CGS) type. The difference with the case dealt with in \cite{hlw} is that the Pohozaev integral $P^{(q)}_r(u)$ is not constant (see \eqref{def:P:invar}). However, it has a finite limit $P^{(q)}(u)$ as $r\to 0$. Therefore, every limiting potential profile $U$ given by Lemma \ref{lem:prelim:1} has a Pohozaev invariant (defined in \eqref{poho:invar:r}) such that $P(U)=P^{(q)}(u)$. It follows from \eqref{pf:2} that $U$ is singular at $0$, and therefore Proposition \ref{prop:poho:U} yields $P(U)>0$. As a consequence, we have that $P^{(q)}(u)>0$. This is enough to make the argument in \cite{hlw} work.

\medskip\noindent All these steps prove Theorem \ref{thm:1}.

\section{The case $q=\crit-1$}\label{sec:q:crit}

The situation here is somehow different since the nonlinearity $u^{\crit-1}$ is invariant after the rescaling performed in Lemma \ref{lem:prelim:1}. We prove the following:
\begin{prop}\label{prop:q:crit} Let $u\in C^\infty(\ballp)$ be a positive solution to
\begin{equation}\label{eq:crit}
-\Delta u= \frac{u^{\crits-1}}{|x|^s}-\mu u^{\crit-1} \hbox{ in }\ballp.
\end{equation}
Then either $0$ is a removable singularity, or there exist $c_1,c_2>0$ such that
\begin{equation}\label{bnd:inf:sup}
c_1|x|^{-\frac{n-2}{2}}\leq u(x)\leq c_2|x|^{-\frac{n-2}{2}}\hbox{ for all }x\in B_{1/2}(0)\setminus\{0\}.
\end{equation}
\end{prop}  

\begin{proof}[Proof of Proposition \ref{prop:q:crit}] We follow the strategy developed in Korevaar--Mazzeo--Pacard--Schoen \cite{kmps} and skip some details. We argue by contradiction and we assume that $0$ is not a removable singularity and that \eqref{bnd:inf:sup} does not hold. By Propositions \ref{prop:bnd:u} and \ref{prop:remove}, it follows that 
\begin{equation}\label{hyp:inf:sup:3}
\liminf_{x\to 0}|x|^{\frac{n-2}{2}}u(x)=0\text{ and }\  \limsup_{x\to 0}|x|^{\frac{n-2}{2}}u(x)\in (0,\infty).
\end{equation}
As in \eqref{def:w}, we define $w(r)=r^{\frac{n-2}{2}}\bar{u}(r)$ for any $r\in (0,1)$. 

\medskip\noindent{\bf Step 1:} We claim that there exists $(t_i)_i\in (0,1/2)$ such that 
\begin{equation}\label{hyp:pt:crit}
\lim_{i\to +\infty}t_i=0,\quad 
\lim_{i\to +\infty}w(t_i)=0\quad \text{and} \quad 
w'(t_i)=0
\end{equation}

\smallskip\noindent{\it Proof of the claim:} Arguing as in the proof of Proposition \ref{prop:construct:radii}, we get that critical points to $r\longmapsto w(r)=r^{\frac{n-2}{2}}\bar{u}(r)$ below a certain threshold are strict local minima. Therefore, if \eqref{hyp:pt:crit} does not hold, then either $w(r)$ stays above a given positive value, or it is monotonic for small $r$, and therefore has a limit as $r\to 0$. These two situations contradict \eqref{hyp:inf:sup:3}. Then there exists $(t_i)_{i}\in (0,1/2)$ such that \eqref{hyp:pt:crit} holds. This proves the claim and ends Step 1.\hfill$\Box$

\medskip\noindent{\bf Step 2:} By defining $v_i(x):=\frac{u(t_i x)}{\bar{u}(t_i)}$ for all $0<|x|<1/t_i$, we claim that
\begin{equation}\label{lim:u:ri}
\lim_{i\to +\infty}v_i(x)=\frac{1}{2}\left(|x|^{2-n}+1\right)\hbox{ in }C^2_{\rm loc}(\rn\setminus\{0\}).
\end{equation}

\medskip\noindent{\it Proof of the claim:}  
Equation \eqref{eq:crit} rewrites as follows
$$-\Delta v_i=w(t_i)^{\crits-2}\frac{v_i^{\crits-1}}{|x|^s}-\mu \, w(t_i)^{\crit-2}v_i^{\crit-1}\quad \text{in } B_{1/t_i}(0)\setminus\{0\}.$$
The Harnack inequality of Lemma \ref{lem:harnack} gives that for any $R>1$, there exist $C_R>0$ and $i_R\in \nn$ such that
\begin{equation}\label{bnd:ui:bis}
1/C_R\leq v_i(x)\leq C_R\hbox{ for all }1/R<|x|<R \hbox{ and }i\geq i_R.
\end{equation}
From \eqref{hyp:pt:crit}, \eqref{bnd:ui:bis} and standard elliptic theory (see for instance \cite{gt}), it follows that there exists $V\in C^2(\rn\setminus\{0\})$ such that
\begin{equation*}
\left\{\begin{array}{ll}
-\Delta V=0 & \hbox{ in }\rn\setminus\{0\}\\
V>0 & \hbox{ in }\rn\setminus\{0\}\\
\lim_{i\to +\infty} v_i=V& \hbox{ in }C^2_{\rm loc}(\rn\setminus\{0\}).
\end{array}\right.
\end{equation*}
By Liouville's theorem, there exist $a,b\geq 0$ such that $V(x)=a|x|^{2-n}+b$ for all $x\in \rnp$. By the mean value theorem, for any $i\in\nn$, there exists $\theta_i\in \partial B_1(0)$ such that $v_i(\theta_i)=1$: taking a subsequence and passing to the limit yields $a+b=1$. Moreover, passing to the limit in the third assumption of \eqref{hyp:pt:crit} yields $(r^{\frac{n-2}{2}}V(r))'(1)=0$, which gives $a=b$. 
This proves \eqref{lim:u:ri}. This ends Step 2.\hfill$\Box$

\medskip\noindent{\bf Step 3:} Here goes the final argument to get the contradiction. Recall the definition of the Pohozaev integral given in \eqref{def:P:invar}:
$$
P_r^{(\crit-1)}(u)=\int_{\partial B_r(0)}\left[(x,\nu)\left(\frac{|\nabla u|^2}{2}-\frac{u^{\crits}}{\crits|x|^s}+\mu\frac{u^{\crit}}{\crit}\right)-T(x,u)\,\partial_\nu u\right]\, d\sigma.
$$
From \eqref{poho:id}, we see that $P_r^{(\crit-1)}(u)$ is independent of $r\in (0,1)$: let $P^{(\crit-1)}(u)$ be the common value.
For $i\in \nn$ and $x\in B_{1/t_i}(0)\setminus\{0\}$, we denote
$$\mathcal P_i(x):= w(t_i)^{\crits-2}\frac{v_i^{\crits}(x)}{\crits|x|^s}-\mu \, w(t_i)^{\crit-2}\frac{v_i^{\crit}(x)}{\crit}. $$
From \eqref{hyp:pt:crit} and the convergence in \eqref{lim:u:ri}, we have $\lim_{i\to +\infty} \mathcal P_i(x)= 0$  uniformly with respect to $x\in \partial B_1(0)$. 
Using a change of variable, we find that
$$ \frac{P_{t_i}^{(\crit-1)}(u)}{w(t_i)^2}
=\int_{\partial B_1(0)}\left[(x,\nu)\left(\frac{|\nabla v_i|^2}{2}-\mathcal P_i(x)\right)- T(x,v_i) \,\partial_\nu v_i \right]\, d\sigma.
$$
Taking the limit $i\to +\infty$ yields
\begin{equation}\label{asymp:ri}
\lim_{i\to +\infty}\frac{P_{t_i}^{(\crit-1)}(u)}{w(t_i)^2}=\frac{(n-2)^2}{8}\omega_{n-1}.
\end{equation}
On the one hand, since $w(t_i)\to 0$ as $i\to +\infty$, we get that $\lim_{i\to +\infty}P_{t_i}^{(\crit-1)}(u)=0$. Therefore $P^{(\crit-1)}(u)=0$. On the other hand, \eqref{asymp:ri} yields $P_{t_i}^{(\crit-1)}(u)>0$ for $i$ large enough, and thus $P^{(\crit-1)}(u)>0$. This is a contradiction. This ends Step 3.

\medskip\noindent Proposition \ref{prop:q:crit} follows from the contradiction obtained in Step 3.\end{proof}

\medskip\noindent As in the case $q\neq\crit-1$, it is natural to investigate more precisely the behavior around $0$, and, hopefully, get a (CGS) profile. The key is to understand the solutions on $\rnp$, which happen to be very sensitive to the choice of the parameter $\mu>0$.

\medskip\noindent We define
\begin{equation*}
\mu_0(n,s):=\frac{(2-s)s^{\frac{s}{2-s}}}{2^{\frac{2(1-s)}{2-s}}(n-2)^{\frac{2s}{2-s}}}\hbox{ and }\mu_1(n,s):=\frac{(2-s)n}{2(n-s)}\left(\frac{2s(n-s)}{n-2}\right)^{\frac{s(n-2)}{2-s}}.
\end{equation*}
As one checks, $0<\mu_1(n,s)<\mu_0(n,s)$.
\begin{prop}[Solutions on $\rnp$]\label{prop:q:crit:lim} For $\mu>\mu_{0}(n,s)$ there is no positive solution to 
\begin{equation}\label{added}
-\Delta u= \frac{u^{\crits-1}}{|x|^s}-\mu u^{\crit-1} \hbox{ in }\rnp.
\end{equation}
For $\mu:=\mu_0(n,s)$, the only positive solution to \eqref{added} is 
$$u(x)=\left(\frac{2-s}{2\mu_0(n,s)}\right)^{\frac{n-2}{2s}}|x|^{-\frac{n-2}{2}}\hbox{ for all }x\in\rnp.$$

\smallskip\noindent When $0<\mu<\mu_0(n,s)$, then for any solution $u$ to \eqref{added}, there exist $c_u,C_u>0$ such that
$$c_u|x|^{-\frac{n-2}{2}}\leq u(x)\leq C_u|x|^{-\frac{n-2}{2}}\hbox{ for all }x\in \rnp.$$
Moreover, any {\it radial} positive solution $u\in C^\infty(\rnp)$ to \eqref{added} is of the form $u(x)=|x|^{-\frac{n-2}{2}}v(-\ln|x|)$ for all $x\in\rnp$, where $v:\rr\to\rr$ is a smooth positive function bounded from above and below by positive constants. In addition, still for radial solutions,
\begin{itemize}
\item If $0<\mu\leq \mu_1(n,s)$, then $v$ is periodic.
\item If $\mu_1(n,s)<\mu<\mu_0(n,s)$, then either $\{ \, v$ is periodic $\}$, or $\{\, v$ is nonconstant with a positive limit as $|x|\to \infty\, \}$.
\end{itemize}
\end{prop}  

\medskip\noindent{\it Proof of Proposition \ref{prop:q:crit:lim}:} We let $u\in C^\infty(\rnp)$ be a positive solution to \eqref{added}. It follows from Proposition \ref{prop:bnd:u} that $u$ is bounded from above by $C |x|^{-\frac{n-2}{2}}$ around $0$. We perform a Kelvin transform on $u$, so that equation \eqref{eq:crit} remains invariant: then the bound $C|x|^{-\frac{n-2}{2}}$ holds around $0$ for the transform of $u$. Going back to $u$, we have the same bound everywhere, so there exists $C>0$ such that
\begin{equation}\label{bnd:u:glob}
u(x)\leq C|x|^{-\frac{n-2}{2}}\hbox{ for all }x\in\rnp.
\end{equation} 
With the conformal map $\varphi$ defined in \eqref{def:phi}, we define
$$v(t,\theta):=e^{-\frac{n-2}{2}t}u(e^{-t}\theta)\hbox{ for all }t\in\rr\hbox{ and }\theta\in\mathbb{S}^{n-1}.$$
With the transformation law \eqref{conf:lap}, the critical equation \eqref{eq:crit} rewrites
\begin{equation}\label{eq:v:3}
-\partial_{tt}v-\Delta_{\hbox{can}_{n-1}}v+ F'(v)=0 \hbox{ in }\rr\times\mathbb{S}^{n-1},
\end{equation}
where 
$$F(v):=\frac{(n-2)^2}{8}v^2+\mu\frac{v^{\crit}}{\crit}-\frac{v^{\crits}}{\crits}.$$
We define $g(v):=v^{-1}F'(v)$ for $v>0$ and $g(0):=\frac{(n-2)^2}{4}$. The function $g$ has a unique critical point, it is decreasing before, and increasing after. As one checks, 
\begin{itemize}
\item If $\mu>\mu_0(n,s)$, then there exists $\eps_0>0$ such that $g(v)\geq \epsilon_0$ for all $v>0$;
\item If $\mu=\mu_0(n,s)$, then $g(v)\geq 0$ for all $v>0$, achieving $0$ only at one point;
\item If $0<\mu<\mu_0(n,s)$, then $\min g<0$ and $g$ vanishes exactly at two points referred to as $v_-<v_+$. In particular, $g'(v_-)<0<g'(v_+)$.
\end{itemize}

\medskip\noindent We assume that $\mu\geq \mu_0(n,s)$. Averaging \eqref{eq:v:3} over $\mathbb{S}^{n-1}$ yields
\begin{equation}\label{eq:average}
-\partial_{tt}\bar{v}+\overline{F'(v(t,\theta))}=0 \hbox{ in }\rr\times\mathbb{S}^{n-1},
\end{equation}
where $\bar{v}(t)$ is the average of $v(t,\theta)$ over $\mathbb{S}^{n-1}$. Since $F'(v(t,\theta))\geq 0$, we get that $\partial_{tt}\bar{v}\geq 0$, and therefore $\bar{v}$ is convex and bounded (this is a consequence of \eqref{bnd:u:glob}), so it is constant. Since $F'\geq 0$ and $\bar{v}$ is constant, \eqref{eq:average} yields $F'(v(t,\theta))\equiv 0$ and then $\mu=\mu_0(n,s)$ and $(t,\theta)\mapsto v(t,\theta)$ is constant equal to the unique zero of $g$. Going back to $u$ yields Proposition \ref{prop:q:crit:lim} for $\mu\geq \mu_0(n,s)$. This ends the proof of Proposition \ref{prop:q:crit:lim}.

\medskip\noindent When $\mu<\mu_0(n,s)$ and $u$ is radially symmetric, the study of $u$ is equivalent to the study of positive solutions $v:\rr\to\rr$ to \eqref{eq:v:3}. The behavior is then a consequence of a classical ODE analysis.\qed

\medskip\noindent As a consequence, we get the following:

\begin{prop}\label{prop:q:crit:2} Let $u\in C^\infty(\ballp)$ be a positive solution to
\begin{equation}
-\Delta u= \frac{u^{\crits-1}}{|x|^s}-\mu u^{\crit-1} \hbox{ in }\ballp.
\end{equation}
If $\mu>\mu_0(n,s)$, then $0$ is a removable singularity. If $\mu=\mu_0(n,s)$, then either $0$ is a removable singularity, or 
$$\lim_{x\to 0}|x|^{\frac{n-2}{2}}u(x)=\left(\frac{2-s}{2\mu_0(n,s)}\right)^{\frac{n-2}{2s}}.$$
\end{prop}  

\smallskip\noindent{\it Proof of Proposition \ref{prop:q:crit:2}:} We assume that $\mu\geq \mu_0(n,s)$ and that $u$ is a solution to the problem with nonremovable singularity. It then follows from Proposition \ref{prop:q:crit} that there exist $c_1,c_2>0$ such that
$$c_1\leq |x|^{\frac{n-2}{2}} u(x)\leq c_2\hbox{ for all }x\in B_{1/2}(0)\setminus \{0\}.$$
We let $(r_i)_i>0$ be any sequence going to $0$, and we define $u_i(x):=r_i^{\frac{n-2}{2}}u(r_i x)$ for all $x\in B_{r_i^{-1}}(0)\setminus \{0\}$. We have that $c_1\leq |x|^{\frac{n-2}{2}} u_i(x)\leq c_2$ for all $x\in B_{r_i^{-1}}(0)\setminus \{0\}$ and $-\Delta u_i=|x|^{-s}u_i^{\crits-1}-\mu u_i^{\crit-1}$ in $B_{r_i^{-1}}(0)\setminus \{0\}$: it then follows from elliptic theory that, up to a subsequence, $u_i\to U$ in $C^2_{\rm loc}(\rnp)$ as $i\to +\infty$. Passing to the limit in the equation yields that $U$ is a positive smooth solution to \eqref{added}. It then follows from Proposition \ref{prop:q:crit:lim} that $\mu=\mu_0(n,s)$ and $U=c|\cdot|^{-\frac{n-2}{2}}$ (for a fixed value $c>0$) is independent of the choice of the sequence $(r_i)_i$. This uniqueness yields Proposition \ref{prop:q:crit:2}.\hfill$\Box$

\end{document}